\documentclass[reqno,11pt]{amsart}
\usepackage[paper=letterpaper,margin=1in]{geometry}

\usepackage{graphicx} 

\usepackage{amsmath,amsfonts,amssymb,amsthm}
\usepackage{xcolor}
\usepackage{MnSymbol}
\usepackage[colorlinks=true, pdfstartview=FitV, urlcolor=blue, citecolor=red, linkcolor=blue]{hyperref}
\usepackage{cleveref}
\usepackage{verbatim}
\usepackage[colorinlistoftodos,prependcaption]{todonotes}
\usepackage{mathbbol}

\numberwithin{equation}{section}

\newtheorem{proposition}{Proposition}[section]
\newtheorem{theorem}{Theorem}[section]
\newtheorem{corollary}{Corollary}[section]
\newtheorem{lemma}{Lemma}[section]

\newtheorem{definition}{Definition}[section]
\newtheorem{remark}{Remark}[section]
\newtheorem{example}{Example}[section]

\newcommand{\WeylCh}{\mathcal{W}}
\newcommand{\diff}{\mathrm{d}}

\newcommand{\A}{\boldsymbol{a}}
\renewcommand{\AA}{\boldsymbol{A}}

\newcommand{\B}{\boldsymbol{b}}
\newcommand{\BB}{\boldsymbol{B}}
\newcommand{\Bbar}{\overline{\B}}

\renewcommand{\L}{\boldsymbol{L}}
\newcommand{\X}{\boldsymbol{x}}
\newcommand{\XX}{\boldsymbol{X}}
\newcommand{\Y}{\boldsymbol{y}}
\newcommand{\Z}{\boldsymbol{z}}
\newcommand{\ZZ}{\boldsymbol{Z}}
\newcommand{\U}{\boldsymbol{u}}
\newcommand{\Zbar}{\overline{\Z}}
\newcommand{\Xbar}{\overline{\X}}
\newcommand{\shlF}{\overline{F}}
\newcommand{\sqwF}{\mathbb{f}}
\newcommand{\sqwFBK}{\mathbb{F}}

\title{Orthogonality of spin $q$-Whittaker polynomials}

\author[M. Mucciconi]{Matteo Mucciconi}
\address[M. Mucciconi]{ 
 Department of Mathematics, National University of Singapore,
 S17, 10 Lower Kent Ridge Road, 119076, Singapore. }
\email{matteomucciconi@gmail.com}
\date{}

\begin{document}

\begin{abstract}
    The inhomogeneous spin $q$-Whittaker polynomials are a family of symmetric polynomials which generalize the Macdonald polynomials at $t=0$. In this paper we prove that they are orthogonal with respect to a variant of the Sklyanin measure on the $n$ dimensional torus and as a result they form a basis of the space of symmetric polynomials in $n$ variables. Instrumental to the proof are inhomogeneous eigenrelations, which partially generalize those of Macdonald polynomials. We also consider several special cases of the inhomogeneous spin $q$-Whittaker polynomials, which include variants of symmetric Grothendieck polynomials or spin Whittaker functions.
\end{abstract}

\maketitle

\tableofcontents

\section{Introduction}

\subsection{Context}

The spin $q$-Whittaker polynomials are a family of symmetric polynomials that can be viewed as a multi-parameter generalization of the $q$-Whittaker functions, which themselves arise as a specialization of the well-known Macdonald polynomials. They were introduced by Borodin and Wheeler \cite{BorodinWheelerSpinq} in connection with integrable lattice models and have been progressively generalized in \cite{MucciconiPetrov2020,Borodin_Korotkikh_Inhom}. In particular, they admit a combinatorial definition as partition functions of integrable vertex models, with Boltzmann weights given (in the most general form) by matrix elements of certain intertwining operators between tensor products of representations of the affine quantum group $U_q'(\widehat{\mathfrak{sl}}_n)$ \cite{Korotkikh2024}. The Yang--Baxter-type relations satisfied by the vertex weights give rise to many remarkable properties of these polynomials, including Cauchy identities, Pieri rules, and integral representations. Such symmetries connect the spin $q$-Whittaker polynomials to solvable particle systems \cite{BufetovMucciconiPetrov2018,MucciconiPetrov2020}, of which they describe the joint probability distribution; see also \cite{Korotkikh_Hidden}. Recently, Korotkikh \cite{Korotkikh2024} provided a characterization of the inhomogeneous spin $q$-Whittaker polynomials, proving that they are the unique symmetric polynomials satisfying a certain family of vanishing conditions; such characterization gives an alternative definition that bypasses the lattice model construction.

In this paper, we prove that the inhomogeneous spin $q$-Whittaker polynomials are self-orthogonal with respect to a scalar product given by an explicit measure on the $n$-dimensional complex torus; see \Cref{thm:orthogonality inhom sqW}. This scalar product is a multi-parameter extension of the ``torus scalar product'' that makes the (classical) $q$-Whittaker polynomials orthogonal \cite{Macdonald1995}, and its density can be viewed as a $q$-deformation of the Sklyanin measure \cite{Sklyanin_The_quantum_toda_Chain}. Our result, conjectured in a special case in \cite{MucciconiPetrov2020}, is in our view rather striking: while symmetries and Cauchy identities follow relatively directly from the underlying Yang--Baxter equation, more subtle properties such as linear independence or orthogonality typically demand different arguments.

The spin $q$-Whittaker polynomials are closely related to another family of symmetric functions, the spin Hall--Littlewood functions \cite{deGierWheeler2016,Borodin2014vertex,BorodinPetrov2016inhom}. Both families admit integrable-lattice-model definitions and satisfy variants of the Cauchy identities \cite{BorodinWheelerSpinq,Borodin_Korotkikh_Inhom}. Moreover, the spin Hall--Littlewood functions are themselves orthogonal under a similar torus scalar product; see \Cref{prop:orthogonality shl}. This orthogonality is deeply connected to the fact that the spin Hall--Littlewood functions form a complete set of Bethe Ansatz eigenfunctions for various self-adjoint operators, including the XXZ spin chain Hamiltonian \cite{BCPS2014}. However, it remains unclear whether the orthogonality of the spin $q$-Whittaker polynomials (which we prove in \Cref{thm:orthogonality inhom sqW}) is of the same nature as that of their spin Hall--Littlewood counterparts.

A classical special case of our orthogonality result is the orthogonality of the Schur polynomials, viewed as the characters of the unitary group $U(n)$; see, e.g., \cite{Bump_Lie_Groups}. Another notable specialization is the orthogonality of the (Archimedean) $GL(n)$ Whittaker functions \cite{Jacquet1967,Kostant_Whittaker}, which is closely connected to integrable Toda lattice theories and can be seen as a manifestation of the Harish-Chandra Plancherel theorem \cite{semenov1994quantization,Kharchev_2001}. A number of $q$-deformations of these classical results have been achieved in \cite{Ruijsenaars_relativistic,Etingof1999,GLO2010,GerasimovLebedevOblezin2011}; see also the more recent survey \cite{Semenov-Tian-Shansky_2023}. Nevertheless, to our knowledge no alternative proof for the orthogonality of $q$-Whittaker case has appeared in the literature. In Macdonald’s original treatment \cite{Macdonald1995}, the “torus” orthogonality property follows quite directly from the eigenvalue relations satisfied by Macdonald polynomials, making the derivation relatively straightforward. In the case of spin $q$-Whittaker polynomials, we establish analogous eigenrelations below (see Section~\ref{sec:eigenrelations}), but these alone do not suffice to characterize the polynomials: indeed, the eigenvalues depend only on the first and last parts of the indexing partitions. Thus the orthogonality of spin $q$-Whittaker polynomials does \emph{not} simply reduce to these eigenrelations, and additional ideas are required. One may speculate that a $t$-deformation of the spin $q$-Whittaker theory could potentially reconcile these issues by bringing it closer to the full Macdonald framework, although no such general theory yet exists.

\subsection{Inhomogeneous spin $q$-Whittaker polynomials} \label{subs:inhom sqW}
In this paper we will deal with multivariate and multi-parametric functions. To denote lists of variables and parameters, we will use boldface font notation as
\begin{equation}
    \Z_N=(z_1,z_2,\dots,z_N).
\end{equation}
For instance for a function of $N$ variables $f(z_1,\dots,z_N)$ we will use the shorthand $f(\Z_N)$. We will also use the convention that
\begin{equation}
    \overline{z}= z^{-1} 
    \qquad
    \text{and}
    \qquad
    \Zbar_N = (\overline{z}_1,\dots,\overline{z}_N) =  (z_1^{-1},\dots,z_N^{-1})
\end{equation}
For a sequence of parameters $\Z_N$ we define the shift operator $\tau$, which drops the first element of the sequence, as
\begin{equation}
    \tau \Z_N = (z_2,\dots,z_N)
\end{equation}
and naturally, for any $k\in \{1,\dots,N-1\}$, we set
\begin{equation} \label{eq:tau k}
    \tau^k \Z_N = (z_{k+1},\dots, z_N).
\end{equation}
At times it will be convenient to not specify the length of a sequence, which we may think as infinite, in which case we drop the subscript from the notation $\Z=(z_1,z_2,\dots)$. Throughout the text, the lists of complex parameters $\A,\B$ will be called \emph{inhomogeneity parameters} and will always be assumed to satisfy
\begin{equation} \label{eq:conditions parameters}
    |a_i|,|b_i| <1 \qquad \text{for all } i\ge 1.
\end{equation}
The conditions in \eqref{eq:conditions parameters} could be relaxed in various results of this paper; however, for the sake of a cleaner presentation, we have opted not to pursue the most general form.

\medskip

The families of symmetric function with which we deal are often labelled by integer partitions, whose set is denoted by $\mathbb{Y}$. We also define $\mathbb{Y}_n$ to be the subset of integer partitions of length at most $n$. The next definition uses the notion of $q$-Pochhammer symbols and interlacing partitions which we recall below in \Cref{subs:notation}.

\begin{definition}
    For $\lambda \in \mathbb{Y}_n$ and $\mu \in \mathbb{Y}_{n-1}$, define the skew spin $q$-Whittaker polynomial ({\' a} la Borodin-Korotkikh \cite{Borodin_Korotkikh_Inhom})
    \begin{equation} \label{eq:skew sqWBK}
            \sqwFBK_{\lambda/\mu} (x| \A_{n}, \B_{n}) = \mathbf{1}_{\mu \preceq \lambda} x^{|\lambda|-|\mu|} \prod_{r = 1}^n \frac{(a_r/x;q)_{\lambda_r-\mu_r} (x b_r;q)_{\mu_r-\lambda_{r+1}} (q;q)_{\lambda_r-\lambda_{r+1}} }{(q;q)_{\lambda_r-\mu_r} (q;q)_{\mu_r-\lambda_{r+1}} (a_{r} b_r;q)_{\lambda_r-\lambda_{r+1}} }.
    \end{equation}
    We define the $n$-variate spin $q$-Whittaker polynomials through branching rules
    \begin{equation} \label{eq:branching rule}
        \sqwFBK_{\lambda} (\X_n|\A_n,\B_n) = \sum_{\mu \preceq \lambda} \sqwFBK_\mu ( \X_{n-1} | \tau \A_n , \B_{n-1} ) \sqwFBK_{\lambda / \mu} (x_n|\A_n,\B_n).
    \end{equation}    
\end{definition}

\begin{remark}
    The polynomials $\sqwFBK_{\lambda} (\X_n|\A_n,\B_n)$ are symmetric in the variables $\X_n$ as proven in \cite{Borodin_Korotkikh_Inhom}; see \Cref{rem:matching notation sqw} to match the above definition with that used in the paper \cite{Borodin_Korotkikh_Inhom}.
\end{remark}

\begin{remark} \label{rem:skew sqWBK}
    The polynomials $\sqwFBK_\lambda$ also have an $n$-variate skew variant, $\sqwFBK_{\lambda/\mu}(\X_n|\A_n,\B_n)$, which can be defined adapting the branching rule \eqref{eq:branching rule}; see \cite{Borodin_Korotkikh_Inhom}. However, since the main results of this paper do not apply to them, we do not discuss them further.    
\end{remark}

In this paper, rather than the Borodin-Korotkikh polynomials $\sqwFBK_\lambda$, we will work with a variant of the spin $q$-Whittaker polynomials defined next.

\begin{definition} \label{def:sqW}
    For $\lambda \in \mathbb{Y}_n$ and $\mu \in \mathbb{Y}_{n-1}$, define
    \begin{equation}
        \sqwF_{\lambda/\mu}(x|\A_{n-1}, \B_{n-1}) = \sqwFBK_{\lambda/\mu}(x|\A_n, \B_n) \big|_{a_n=0}
    \end{equation}    
    and
    \begin{equation}
        \sqwF_\lambda(\X_n|\A_{n-1}, \B_{n-1}) = \sqwFBK_\lambda(\X_n|\A_n, \B_n) \big|_{a_n=0}.
    \end{equation}
\end{definition}

\begin{remark}
    Notice that $\sqwFBK_\lambda(\X_n|\A_n, \B_n)\big|_{a_n=0}$, as a function of $\B_n$ only depends of $b_1,\dots,b_{n-1}$, justifying the notation adopted for the functions $\sqwF$. This is because the $n$-th term in the product in the right hand side of \eqref{eq:skew sqWBK} is
    \begin{equation}
        \frac{(a_n/x;q)_{\lambda_n-\mu_n} (x b_n;q)_{\mu_n-\lambda_{n+1}} (q;q)_{\lambda_n-\lambda_{n+1}} }{(q;q)_{\lambda_n-\mu_n} (q;q)_{\mu_n-\lambda_{n+1}} (a_n b_n;q)_{\lambda_n-\lambda_{n+1}} } = \frac{(a_n/x;q)_{\lambda_n} }{ (a_n b_n;q)_{\lambda_n} },
    \end{equation}
    since, by assumptions, $\mu_n=\lambda_{n+1}=0$.
\end{remark}

At first glance, the \Cref{def:sqW} of the functions $\sqwF$ might seem irrelevant. However, setting the inhomogeneity parameter $a_n$ to zero endows the $n$-variate spin $q$-Whittaker polynomials with additional symmetries and desirable properties. In this paper, we collect several such properties, including the \emph{shift property} (\Cref{prop:shift property}), the \emph{reverse symmetry} (\Cref{prop:reverse symmetry}), and the eigenrelations (\Cref{thm:eigenrelation sqW lambda n}, \Cref{thm:eigenrelation sqW lambda 1}).
We will also see, in \Cref{prop:relation sqW and sqWBK}, that one can recover the polynomials $\sqwFBK$ as a special case of $\sqwF$ as
\begin{equation}
    \sqwFBK_\lambda(\X_n|\A_n, \B_n) = \sqwF_\lambda(\X_{n+1}|\A_n,\B_n) \big |_{x_{n+1}=0},
\end{equation}
so neither family of polynomials $\{ \sqwF_\lambda (\X_n|\A_{n-1},\B_{n-1}) \}$, $\{ \sqwFBK_\lambda (\X_n| \A_n,\B_n) \}$ is more general than the other if we allow the parameter $n$ counting the number of variables to vary.

\subsection{Main results}
Let $\mathbb{T}$ denote the complex circle $\{ z \in \mathbb{C} : |z|=1 \}$. For $\Z_n\in \mathbb{T}^n$ and parameters $\A_{n-1},\B_{n-1}$, as in \eqref{eq:conditions parameters} we define the function
\begin{equation} \label{eq: delta A B}
    \Delta (\Z_n|\A_{n-1},\B_{n-1}) = \frac{ H_q( \Zbar_n  ; \A_{n-1}) H_q( \Z_n ; \B_{n-1} ) \prod_{k=1}^{n-1} (a_k b_k;q)_\infty  }{ H_q( \A_{n-1} ; \B_{n-1} ) } \Delta (\Z_n),
\end{equation}
where
\begin{equation} \label{eq:H}
    H_q(\X_N; \Y_M) = \prod_{i=1}^N \prod_{j=1}^M \frac{1}{(x_i y_j;q)_\infty}
\end{equation}
and
\begin{equation} \label{eq:Delta}
    \Delta(\Z_n) = \frac{1}{n!} \prod_{1 \le i \neq j \le n} (z_i/z_j;q)_\infty \prod_{i=1}^n \frac{1}{2 \pi \mathrm{i} z_i}.
\end{equation}
The function $\Delta(\Z_n|\A_{n-1}, \B_{n-1})$ is the density function of a measure on the $n$-dimensional torus and we denote the corresponding scalar product as
\begin{equation} \label{eq:sqW scalar product}
    \llangle f \,|\, g \rrangle_n = \int_{\mathbb{T}^n} f(\Z_n) g(\Zbar_n) \Delta (\Z_n|\A_{n-1},\B_{n-1}) \mathrm{d} z_1 \cdots \mathrm{d} z_n.
\end{equation}
With a slight abuse of notation we will at times keep the dependence on the variables $\Z_n$ of functions $f,g$ as
\begin{equation}
    \llangle f(\Z_n) \,|\, g(\Z_n) \rrangle_n = \llangle f \,|\, g \rrangle_n.
\end{equation}
Since the torus integration variables will always be denoted by $\Z$, we hope that this causes no confusion.

\medskip

Our main result states that the spin $q$-Whittaker polynomials $\sqwF$ are self-orthogonal with respect to the scalar product $\llangle\,| \, \rrangle$.
\begin{theorem} \label{thm:orthogonality inhom sqW}
    Fix $n \in \mathbb{N}$. For all $\lambda,\mu \in \mathbb{Y}_n$, we have
    \begin{equation} \label{eq:orthogonality sqW}
        \left\llangle \sqwF_\lambda(\Z_n|\A_{n-1},\B_{n-1}) \middle| \sqwF_\mu (\Z_n| \B_{n-1} , \A_{n-1} ) \right\rrangle_n = \mathbf{1}_{\lambda=\mu} \mathsf{c}_n(\lambda | \A_{n-1}, \B_{n-1} ) ,
    \end{equation}
    where
    \begin{equation} \label{eq:c AB}
        \mathsf{c}_n(\lambda | \A_{n-1}, \B_{n-1} ) = \prod_{k=1}^{n-1} \frac{(a_k b_k;q)_\infty}{(q;q)_\infty} \frac{(q;q)_{\lambda_k-\lambda_{k+1}}}{(a_k b_k;q)_{\lambda_k-\lambda_{k+1}} }.
    \end{equation}
\end{theorem}

Immediate important consequences of \Cref{thm:orthogonality inhom sqW} are given next. In the following we denote by $\mathsf{Sym}_n$ the linear space of symmetric polynomials in $n$ variables.

\begin{theorem} \label{thm:sqW are a basis}
    Fix $n \in \mathbb{N}$. The family of symmetric polynomials $\{ \sqwF_\lambda(\X_n|\A_{n-1},\B_{n-1}) \}_{\lambda \in \mathbb{Y}_n}$ is a linear basis of $\mathsf{Sym}_n$.
\end{theorem}

\begin{theorem} \label{thm:sqWBK are a basis}
    Fix $n \in \mathbb{N}$. The family of symmetric polynomials $\{ \sqwFBK_\lambda(\X_n|\A_{n},\B_{n}) \}_{\lambda \in \mathbb{Y}_n}$ is a linear basis of $\mathsf{Sym}_n$.
\end{theorem}

\begin{remark}
    In \cite{Korotkikh2024}, the polynomials $\sqwFBK_\lambda(\X_n|\A_{n},\B_{n})$ were shown to be the the unique family of symmetric polynomials satisfying certain vanishing conditions. As a result they were proven to be a basis of the space
    \begin{equation}
        \mathrm{Span} \left\{ G_\mu(\X_n) \right\}_{\mu \in \mathbb{Y}_n},
        \qquad
        \text{with}
        \qquad
        G_\mu(\X_n) = \prod_{i=1}^n \prod_{r \ge 1} \frac{(x_i b_r ;q)_{\mu_r-\mu_{r+1} } }{ (a_i b_r ;q)_{\mu_r-\mu_{r+1} } }.
    \end{equation}
    Such result would give an alternative proof of \Cref{thm:sqW are a basis} and  \Cref{thm:sqWBK are a basis} (but not of \Cref{thm:orthogonality inhom sqW}).    
\end{remark}

In \Cref{sec:special cases} we consider special cases of the spin $q$-Whittaker polynomials and the corresponding specialization of \Cref{thm:orthogonality inhom sqW}. Setting $q=a_1=\cdots=a_{n-1}=0$, the polynomials $\sqwF$ become an inhomogeneous variant of the symmetric Grothendieck polynomials, which arise in the study of K-theory of Grassmannians \cite{Lascoux_Schutzemberger_Grothendieck,Fomin_Kirillov_Grothendieck,Buch_Grothendieck}. In this case the orthogonality relation \eqref{eq:orthogonality sqW} turns into a torus scalar product orthonormality for families of symmetric Grothendieck polynomials, which appears to be new; see \Cref{thm:orthogonality Grothendieck polys}.

\subsection{Layout of the paper} In \Cref{sec:symmetric functions} we define the families of symmetric functions that we use throughout the paper, stating some of their (known or new) properties. In \Cref{sec:eigenrelations} we prove that the spin $q$-Whittaker polynomials are eigenvectors of certain $q$-difference operators. In \Cref{sec:proofs}, using results of the previous Sections, we give the proofs of \Cref{thm:orthogonality inhom sqW}, \Cref{thm:sqW are a basis} and \Cref{thm:sqWBK are a basis}. In \Cref{sec:special cases} we consider three particular cases of the spin $q$-Whittaker polynomials, the symmetric Grothendieck polynomials, the interpolation $q$-Whittaker polynomials and the spin Whittaker functions, specializing the result of \Cref{thm:orthogonality inhom sqW}. We conclude in \Cref{sec:conclusion}, addressing the original motivation for \Cref{thm:orthogonality inhom sqW} and connections with integrable probabilistic models.

\subsection{Notation} \label{subs:notation}

We have already defined the set of integer partitions $\mathbb{Y}$ and the set of integer partitions with length at most $n$ denoted by $\mathbb{Y}_n$. A partition $\lambda \in \mathbb{Y}_n$ can be thought as a list $\lambda=\lambda_1 \ge \cdots \ge \lambda_n \ge 0$ or as a list of arbitrary length appending a number of $0$'s after the $n$-th entry. We will denote
\begin{equation}
    |\lambda| = \lambda_1 +\lambda_2 + \cdots.
\end{equation}
The interlacing relation $\preceq$ between partitions is defined as
\begin{equation}
    \mu \preceq \lambda \qquad \Leftrightarrow  \qquad \lambda_i \ge \mu_i \ge \lambda_{i+1} \qquad \text{for all }i \ge 1.
\end{equation}
If $\lambda \in \mathbb{Y}$, its transposed partition is denoted by $\lambda'$ and it is defined as $\lambda_i'=\# \{ j : \lambda_j \ge i \}$. We will also use the notion of \emph{boxed partitions}, which are partitions $\lambda$ with bounded length and bounded first row. 
For any $m,n \in \mathbb{Z}_{\ge 0}$ we set
\begin{equation} \label{eq:box}
    \mathrm{Box}(\ell,m) = \left\{ \lambda \in \mathbb{Y} : \lambda_1 \le \ell \, \text{and } \, \ell(\lambda) \le m \right\}.
\end{equation}
In this paper, integer partitions are (totally) ordered according to the reverse lexicographic order; thus, we write $\mu \le \lambda$ if and only if $\mu=\lambda$ or if the first non zero difference $\lambda_i - \mu_i$ is positive. 

\medskip

Throughout, $q$ is a complex parameter such that $|q|<1$. The $q$-Pochhammer symbols are defined as
\begin{equation}
    (z;q)_n = 
    \begin{cases}
        \prod_{i=0}^{n-1} (1-zq^i) \qquad & \text{for } n > 0,
        \\
        1 \qquad & \text{for } n = 0,
        \\
        \prod_{i=n}^{-1} (1-zq^i)^{-1} \qquad & \text{for } n < 0
    \end{cases}
\end{equation}
and the above definition extends to the case when $n=+\infty$.

\subsection{Acknowledgements} I am grateful to Alexei Borodin, Sergei Korotkikh, Leonid Petrov and Travis Scrimshaw for comments on a preliminary version of the draft.

\section{Symmetric functions} \label{sec:symmetric functions}

In the text, at time, we will use the Macdonald polynomials \cite[VI]{Macdonald1995} commonly denoted by $P_{\lambda}(\X_n;q,t)$ and $Q_{\lambda}(\X_n;q,t)$, which depend on a set of $n$ variables $\X_n$ and two parameters $t,q$ such that $|t|,|q|<1$. Setting $t=0$ we recover the $q$-Whittaker polynomials $P_\lambda(\X_n;q,0)$ and $Q_\lambda(\X_n;q,0)$. Setting $q=0$ we recover the Hall-Littlwood polynomials $P_\lambda(\X_n;0,t)$ and $Q_\lambda(\X_n;0,t)$, although below we will still use $q$ in place of $t$ to denote their deformation parameter.

\subsection{Properties of inhomogeneous spin $q$-Whittaker polynomials}

In \Cref{subs:inhom sqW} we defined the two families of spin $q$-Whittaker polynomials $\sqwFBK$ and $\sqwF$. Since in the text we will make use of results in literature, let us compare our definition of the functions $\sqwFBK$ and $\sqwF$ with those of \cite{Borodin_Korotkikh_Inhom,Korotkikh2024}.

\begin{remark} \label{rem:matching notation sqw}
    The functions $\sqwFBK$ correspond to those in \cite{Korotkikh2024} , which we denote here as $\mathbb{F}^{\mathrm{Kor}}$, by
    \begin{equation}
        \sqwFBK_\lambda(\X_n|\A_n,\Bbar_n) = \left( \prod_{r \ge 1} b_{r}^{\lambda_{r+1}} \frac{ (q;q)_{\lambda_r -\lambda_{r+1}} }{ (a_r b_r;q)_{\lambda_r -\lambda_{r+1}} } \right) \mathbb{F}^{\mathrm{Kor}}_\lambda (\X_n|(a_0,a_1,\dots), (b_0,b_1,\dots) ) \big|_{b_0=1}.
    \end{equation}
    They can also be matched with those in \cite{Borodin_Korotkikh_Inhom}, which we denote here as $\mathbb{F}^{\mathrm{BK}}$, by
    \begin{equation}
        \sqwFBK_\lambda(\X_n|\A_n,\Bbar_n) \big|_{a_i=\xi_i s_i , \, b_i= \xi_i/s_i }  =  (-1)^{|\lambda|}\left( \prod_{i\ge 1} \xi_{i}^{\lambda_{i+1}} \right) \mathbb{F}_\lambda^{\mathrm{BK}} (\X_n | \Xi,\mathsf{S}) \big|_{\xi_0=1}.
    \end{equation}
\end{remark}

As stated in \Cref{subs:inhom sqW}, one can recover the family of polynomials $\sqwFBK$ as a special case of the family $\sqwF$ as prescribed in the following proposition.

\begin{proposition} \label{prop:relation sqW and sqWBK}
    We have
    \begin{equation} \label{eq:from sqW to sqWBK}
        \sqwFBK_\lambda(\X_n|\A_n, \B_n) = \sqwF_\lambda(\X_{n+1}|\A_n,\B_n) \big |_{x_{i}=0},
    \end{equation}
    for any $i=1,\dots,n+1$.
\end{proposition}

\begin{proof}
    It suffices to prove \eqref{eq:from sqW to sqWBK} for $i=1$, since the polynomial $\sqwF_\lambda(\X_{n+1}|\A_n,\B_n)$ is symmetric in the $\X_{n+1}$ variables.
    Notice that we can isolate the factors depending on the inhomogeneities $a_n$ (and $b_n$) in the expression of the skew spin $q$-Whittaker polynomial \eqref{eq:skew sqWBK} as
    \begin{equation}
        \sqwFBK_{\lambda / \mu}(x| \A_{n}, \B_{n}) = \sqwF_{\lambda / \mu}(x| \A_{n-1}, \B_{n-1}) \frac{(a_n/x;q)_{\lambda_n}}{(a_n b_n;q)_{\lambda_n}}
    \end{equation}
    and therefore the branching factors $\sqwFBK_{\lambda / \mu}(x| \A_{n}, \B_{n})$ and $\sqwF_{\lambda / \mu}(x| \A_{n-1}, \B_{n-1})$ are equal whenever $\lambda_n=0$. Moreover notice that
    \begin{equation}
        \sqwFBK_{\lambda / \mu}(x| \A_{n}, \B_{n}) = \sqwF_{\lambda / \mu}(x| \A_{n-1}, \B_{n-1}) = 0 \qquad \text{if } \mu_{n-1} = 0, \text{ and } \lambda_n>0,
    \end{equation}
    which is due to the presence of the term $1/(q;q)_{\mu_{n-1}-\lambda_{n}}$ in \eqref{eq:skew sqWBK}. These observations allow us to prove \eqref{eq:from sqW to sqWBK} by induction. For $n=1$, we have
    \begin{equation} \label{eq:one variable sqW}
        \sqwF_{(k)}(x_1|\varnothing, \varnothing) = x_1^k, 
    \end{equation}
    which implies 
    \begin{equation}
        \sqwF_{(k,0)}(0,x_2|a_1, b_1) = x_2^k \frac{(a_1/x_2;q)_k}{(a_1 b_1;q)_k} = \sqwFBK_{(k)}(x_2|a_1, b_1),
    \end{equation}
    verifying \eqref{eq:from sqW to sqWBK}. By inductive hypothesis, assume that 
    \begin{equation}
         \sqwF_\lambda(\X_{n}|\A_{n-1},\B_{n-1}) \big |_{x_{1}=0} = 
         \begin{cases} 
            \sqwFBK_\lambda(x_2,\dots,x_n|\A_{n-1}, \B_{n-1}) \qquad & \text{if } \ell(\lambda) \le n-1,
            \\
            0 \qquad & \text{otherwise}.
         \end{cases}
    \end{equation}
    Then, for all $\lambda \in \mathbb{Y}_n$
    \begin{equation}
        \begin{split}
            \sqwF_\lambda(\X_{n+1}|\A_n,\B_n) \big |_{x_{1}=0} & = \sum_{ \mu \in \mathbb{Y}_{n-1} } \sqwF_\mu(\X_{n}|\tau\A_n,\B_{n-1}) \big |_{x_{1}=0} \sqwF_{\lambda/\mu}(x_{n+1}|\A_n,\B_n)
            \\
            & = \sum_{\mu} \sqwFBK_\mu(x_2,\dots,x_n|\tau\A_n,\B_{n-1}) \sqwFBK_{\lambda/\mu}(x_{n+1}|\A_n,\B_n)
            \\
            & = \sqwFBK_\lambda(x_2,\dots,x_{n+1}|\A_{n},\B_n),
        \end{split}
    \end{equation}
    completing the proof.
\end{proof}

In the following proposition, we state an important structural property of the spin $q$-Whittaker polynomials $\sqwF$, which is \emph{not} satisfied by the family $\sqwFBK$ and which in the context of symmetric functions is commonly referred to as \emph{shift property}. The shift property is a key ingredient, for instance, in the proof of the orthogonality of \Cref{thm:orthogonality inhom sqW}. 

\begin{proposition} \label{prop:shift property}
    The spin $q$-Whittaker polynomials satisfy the shift property: for any $\lambda$ such that $\lambda_n>0$, we have
    \begin{equation} \label{eq:shift property}
        \sqwF_{(\lambda_1,\dots, \lambda_n)}(\X_n|\A_{n-1},\B_{n-1}) = (x_1 \cdots x_n)  \sqwF_{(\lambda_1-1,\dots, \lambda_n-1)} (\X_n|\A_{n-1},\B_{n-1}).
    \end{equation}
\end{proposition}

\begin{proof}
    Notice that the branching factor $\sqwF_{\lambda/\mu}(x|\A_{n-1},\B_{n-1})$ satisfies the shift property
    \begin{equation} \label{eq: shift property branching}
        \sqwF_{\lambda/\mu}(x|\A_{n-1},\B_{n-1}) = x \, \sqwF_{(\lambda_1 -1 ,\dots, \lambda_n -1 )/(\mu_1-1 ,\dots, \mu_{n-1}-1)}(x|\A_{n-1},\B_{n-1}) 
    \end{equation}
    for all pairs of partitions $\lambda,\mu$ such that $\lambda_n,\mu_{n-1}>0$. Iterating the equality \eqref{eq: shift property branching} across all terms of the branching expansion \eqref{eq:branching rule}, we prove \eqref{eq:shift property}.
\end{proof}

\begin{example}
    The shift property does not hold for the polynomials $\sqwFBK(\X_n|\A_n, \B_n)$, which can be seen by simply analyzing the case $n=1$: for any $k\ge 0$, we have
    \begin{equation}
        \sqwFBK_{(k)}(x|a_1,b_1) = x^{k} \frac{(a_1/x;q)_{k} }{ (a_{1} b_1;q)_{k} }.
    \end{equation}
    On the other hand the one-variable polynomials $\sqwF_{(k)}(x|\varnothing,\varnothing)$ were given in \eqref{eq:one variable sqW}.
\end{example}

\begin{remark}
    In the case $q=a_1=b_1=\cdots=0$, the spin $q$-Whittaker polynomials $\sqwF_\lambda$ reduce to the Schur polynomials $s_\lambda(\X_n)$. In this case the shift property reflects the fact that tensoring an irreducible polynomial representation of $GL(n,\mathbb{C})$ with the 1-dimensional determinant representation shifts its highest weight by $(1,\dots,1)$.
\end{remark}

\begin{remark}
    Notice that \Cref{prop:relation sqW and sqWBK} and \Cref{prop:shift property} imply that $\sqwFBK_\lambda (\X_n|\A_n,\B_n)= 0$, whenever $\ell(\lambda)=n+1$, but they do not imply that $\sqwFBK_\lambda (\X_n|\A_n,\B_n)= 0$ if $\ell(\lambda)\le n$.
\end{remark}

The shift property \eqref{eq:shift property} can be used to extend the spin $q$-Whittaker polynomials $\sqwF_\lambda$ to Laurent polynomials labeled by signatures 
\begin{equation}
    \mathrm{Sig}_{n}=\{\lambda \in \mathbb{Z}^n \, : \, \lambda_1\ge \cdots \ge \lambda_n \},
\end{equation} 
through the relation
    \begin{equation} \label{eq:sqW for signatures}
        \sqwF_\lambda(\X_n|\A_{n-1},\B_{n-1}) = (x_1 \cdots x_n)^{\lambda_n} \sqwF_{(\lambda_1-\lambda_n, \dots , \lambda_{n-1} - \lambda_n , 0)}(\X_n|\A_{n-1},\B_{n-1}).
    \end{equation}
    Moreover, for any pair of signatures $\lambda\in \mathrm{Sig}_n$, $\mu \in \mathrm{Sig}_{n-1}$, the single variable skew spin $q$-Whittaker (Laurent) polynomials can be defined as 
    \begin{equation} \label{eq:skew sqW}
        \sqwF_{\lambda/\mu} (x| \A_{n-1}, \B_{n-1}) = \mathbf{1}_{\mu \preceq \lambda} x^{|\lambda|-|\mu|} \prod_{r = 1}^{n-1} \frac{(a_r/x;q)_{\lambda_r-\mu_r} (x b_r;q)_{\mu_r-\lambda_{r+1}} (q;q)_{\lambda_r-\lambda_{r+1}} }{(q;q)_{\lambda_r-\mu_r} (q;q)_{\mu_r-\lambda_{r+1}} (a_{r} b_r;q)_{\lambda_r-\lambda_{r+1}} }
    \end{equation}
    and it is clear that the Laurent polynomials $\sqwF_{\lambda}$ inherit the branching rule
    \begin{equation}
        \sqwF_{\lambda} (\X_n|\A_{n-1},\B_{n-1}) = \sum_{ \substack{  \mu \in \mathrm{Sig}_{n-1} \\
        \mu \preceq \lambda  } }\sqwF_\mu ( \X_{n-1} | \tau \A_{n-1} , \B_{n-2} ) \sqwF_{\lambda / \mu} (x_n|\A_{n-1},\B_{n-1}).
    \end{equation}
    Above, the interlacing relation between signatures $\mu , \lambda$ of length respectively $n-1$ and $n$ is defined as
    \begin{equation}
        \mu \preceq \lambda \qquad \Leftrightarrow \qquad \lambda_i \ge \mu_i \ge \lambda_{i+1} \qquad \text{for all } i=1,\dots,n-1
    \end{equation}
    and symbol $|\lambda|$ is interpreted as
    \begin{equation}
        |\lambda| = \lambda_1 + \cdots + \lambda_n.
    \end{equation}

\begin{remark}
    In light of \Cref{thm:orthogonality inhom sqW} and \Cref{thm:sqW are a basis}, the family $\{ \sqwF_\lambda(\X_n|\A_{n-1},\B_{n-1}) \}_{\lambda\in \mathrm{Sig}_n}$ forms an orthogonal basis of the linear space of symmetric Laurent polynomials in $n$ variables. Notice that such statement cannot be immediately extended to the polynomials $\sqwFBK$, since there is no automatic way of extending them to Laurent polynomials labeled by signatures.
\end{remark}

Extending the definition of spin $q$-Whittaker polynomials to arbitrary signatures as in \eqref{eq:sqW for signatures}, we gain the following symmetry.

\begin{proposition} \label{prop:reverse symmetry}
    For every signature $\lambda=(\lambda_1 \ge \cdots \ge \lambda_n)$ we have
    \begin{equation} \label{eq:reverse symmetry}
        \sqwF_\lambda(\X_n|\A_{n-1},\B_{n-1}) = \sqwF_{\lambda^{\mathrm{rev}}}(\Xbar_n|\B_{n-1}^{\leftarrow},\A_{n-1}^{\leftarrow}),
    \end{equation}
    where
    \begin{equation}
        \lambda^{\mathrm{rev}} = (-\lambda_n \ge \cdots \ge -\lambda_1),
        \qquad
        \A_{n-1}^\leftarrow = (a_{n-1},\dots,a_1), \qquad
        \B_{n-1}^\leftarrow = (b_{n-1},\dots,b_1).
    \end{equation}
\end{proposition}

\begin{proof}
    The single variable skew spin $q$-Whittaker (Laurent) polynomials satisfy the relation
    \begin{equation}
        \sqwF_{\lambda/\mu}(x|\A_{n-1},\B_{n-1}) = \sqwF_{\lambda^{\mathrm{rev}} / \mu^{\mathrm{rev}}}(\overline{x}|\B_{n-1}^{\leftarrow},\A_{n-1}^{\leftarrow}),
    \end{equation}
    as it can be readily verified from \eqref{eq:skew sqW}. 
    For instance, in the case $n=3$, for a pair of signatures $\mu \preceq \lambda$, we have
    \begin{equation}
        \begin{split}
            &\sqwF_{\lambda/\mu} (x| \A_{2}, \B_{2}) 
            \\
            &= x^{|\lambda|-|\mu|} \frac{(a_1/x;q)_{\lambda_1-\mu_1} (x b_1;q)_{\mu_1-\lambda_{2}} (q;q)_{\lambda_1-\lambda_{2}} }{(q;q)_{\lambda_1-\mu_1} (q;q)_{\mu_1-\lambda_{2}} (a_{1} b_1;q)_{\lambda_1-\lambda_{2}} } \frac{(a_2/x;q)_{\lambda_2-\mu_2} (x b_2;q)_{\mu_2-\lambda_{3}} (q;q)_{\lambda_2-\lambda_{3}} }{(q;q)_{\lambda_2-\mu_2} (q;q)_{\mu_2-\lambda_{3}} (a_{2} b_2;q)_{\lambda_2-\lambda_{3}} }
            \\
            &
            = (\overline{x})^{-|\lambda|+|\mu|}  \frac{ (b_2/\overline{x};q)_{\mu_2-\lambda_{3}} (a_2 \overline{x};q)_{\lambda_2-\mu_2} (q;q)_{\lambda_2-\lambda_{3}} }{  (q;q)_{\mu_2-\lambda_{3}} (q;q)_{\lambda_2-\mu_2} (a_{2} b_2;q)_{\lambda_2-\lambda_{3}} } \frac{  (b_1/\overline{x};q)_{\mu_1-\lambda_{2}} (a_1 \overline{x};q)_{\lambda_1-\mu_1} (q;q)_{\lambda_1-\lambda_{2}} }{ (q;q)_{\mu_1-\lambda_{2}} (q;q)_{\lambda_1-\mu_1} (a_{1} b_1;q)_{\lambda_1-\lambda_{2}} }
            \\
            &
            = (\overline{x})^{|\lambda^{\mathrm{rev}}|-|\mu^{\mathrm{rev}}|}  \frac{ (b_2/\overline{x};q)_{ \lambda^{\mathrm{rev}}_{1} - \mu^{\mathrm{rev}}_1 } (a_2 \overline{x};q)_{ \mu^{\mathrm{rev}}_1 - \lambda^{\mathrm{rev}}_2 } (q;q)_{ \lambda^{\mathrm{rev}}_{1} - \lambda^{\mathrm{rev}}_2 } }{  (q;q)_{ \lambda^{\mathrm{rev}}_{1} - \mu^{\mathrm{rev}}_1 } (q;q)_{ \mu^{\mathrm{rev}}_1 - \lambda^{\mathrm{rev}}_2 } (a_{2} b_2;q)_{ \lambda^{\mathrm{rev}}_{1} - \lambda^{\mathrm{rev}}_2 } }
            \\
            & \qquad \qquad \qquad \qquad \qquad \qquad \times \frac{  (b_1/\overline{x};q)_{\lambda^{\mathrm{rev}}_{2} - \mu^{\mathrm{rev}}_2  } (a_1 \overline{x};q)_{ \mu^{\mathrm{rev}}_2 - \lambda^{\mathrm{rev}}_3} (q;q)_{\lambda^{\mathrm{rev}}_2-\lambda^{\mathrm{rev}}_{3}} }{ (q;q)_{\lambda^{\mathrm{rev}}_{2} - \mu^{\mathrm{rev}}_2  } (q;q)_{ \mu^{\mathrm{rev}}_2 - \lambda^{\mathrm{rev}}_3} (a_{1} b_1;q)_{\lambda^{\mathrm{rev}}_2-\lambda^{\mathrm{rev}}_{3}} }
            \\
            & = \sqwF_{\lambda^{\mathrm{rev}} / \mu^{\mathrm{rev}}}(\overline{x}|\B_{2}^{\leftarrow},\A_{2}^{\leftarrow}).
        \end{split}
    \end{equation}
    Assuming that \eqref{eq:reverse symmetry} holds for $(n-1)$-variate functions and using the branching rule we can inductively compute
    \begin{equation}
        \begin{split}
            \sqwF_{\lambda} (\X_n|\A_{n-1},\B_{n-1}) &= \sum_{\mu \preceq \lambda} \sqwF_\mu ( \X_{n-1} | \tau \A_{n-1} , \B_{n-2} ) \sqwF_{\lambda / \mu} (x_n|\A_{n-1},\B_{n-1})
            \\
            & = \sum_{\mu \preceq \lambda} \sqwF_{\mu^{\mathrm{rev}}} ( \Xbar_{n-1} | \B_{n-2}^{\leftarrow}, (a_{n-1}, \dots a_{2} ) ) \sqwF_{\lambda^{\mathrm{rev}} / \mu^{\mathrm{rev}}}(\overline{x}|\B_{n-1}^{\leftarrow},\A_{n-1}^{\leftarrow})
            \\
            & = \sum_{\mu \preceq \lambda} \sqwF_{\mu^{\mathrm{rev}}} ( \Xbar_{n-1} | \tau(\B_{n-1}^{\leftarrow}),  (a_{n-1}, \dots a_{2} ) ) \sqwF_{\lambda^{\mathrm{rev}} / \mu^{\mathrm{rev}}}(\overline{x}|\B_{n-1}^{\leftarrow},\A_{n-1}^{\leftarrow})
            \\
            & = \sqwF_{\lambda^{\mathrm{rev}}}(\Xbar_n|\B_{n-1}^{\leftarrow},\A_{n-1}^{\leftarrow}),
        \end{split}
    \end{equation}
    completing the proof.
\end{proof}

\subsection{Cauchy identities}

As proven in \cite{BorodinWheelerSpinq,MucciconiPetrov2020,Borodin_Korotkikh_Inhom,Korotkikh2024}, the spin $q$-Whittaker polynomials satisfy summation identities generalizing the Cauchy identities for Schur polynomials. To state them we need the following.

\begin{definition}
    The \emph{dual spin $q$-Whittaker polynomial} is
    \begin{equation} \label{eq:dual sqWBK}
        \sqwFBK^*_{\lambda} (\Y_m |\A_m, \B_m) = \psi_\lambda(\A_m, \B_m) \sqwFBK_{\lambda} (\Y_m|\A_m, \B_m),
    \end{equation}
    where
    \begin{equation} \label{eq:psi}
        \psi_\lambda(\A_m,\B_m) = \prod_{r = 1}^m \frac{(a_r b_r;q)_{\lambda_r - \lambda_{r+1}}}{(q;q)_{\lambda_r - \lambda_{r+1}}}.
    \end{equation}
    Notice that $\psi_\lambda(\A_m,\B_m) = \psi_\lambda(\B_m,\A_m)$.
\end{definition}

\begin{remark} \label{eq:limit sqW to qW}
    A particular case of the spin $q$-Whittaker polynomials are the $q$-Whittaker polynomials, which are obtained as
        \begin{equation} \label{eq:sqW at q=0}
            P_\lambda(\X_n;q,0)= \sqwFBK_\lambda(\X_n | \A_n,\B_n) \big|_{ a_i= b_i = 0 },
            \qquad
            Q_\lambda(\X_n;q,0)=   \sqwFBK_\lambda^*(\X_n | \B_n ,\A_n) \big|_{ a_i= b_i = 0 }.
        \end{equation}
\end{remark}

\begin{proposition}[\cite{Korotkikh2024}, Theorem B]
    The spin $q$-Whittaker polynomials satisfy the Cauchy identity
    \begin{equation} \label{eq:inhom CI sqWBK sqWBK}
        \sum_\lambda \sqwFBK_\lambda (\X_n| \A_n, \B_n) \sqwFBK_\lambda^* (\Y_m| \B_m , \A_m ) = \frac{H_q(\X_n;\Y_m) H_q(\A_n;\B_m)}{ H_q(\X_n;\B_m) H_q(\A_n;\Y_m) },
    \end{equation}
    where the function $H_q$ was defined in \eqref{eq:H}.
\end{proposition}

Identity \eqref{eq:inhom CI sqWBK sqWBK} was proved in the original paper \cite{BorodinWheelerSpinq} in the homogeneous case $a_1=a_2=\cdots=b_1=b_2=\cdots=-s$ as a consequence of the Yang-Baxter equation for the six vertex model. Interestingly, the proof of the fully inhomogeneous Cauchy identity is more involved and it is not proven in the same way, but through inhomogeneous intertwining relations (which are different from the Yang-Baxter equation) discovered in \cite{Korotkikh2024}.

\begin{corollary}
    Fix $0<m<n$. We have
    \begin{equation} \label{eq:inhom CI sqW sqW}
        \sum_\lambda \sqwF_\lambda (\X_n| \A_{n-1}, \B_{n-1}) \sqwFBK_\lambda^* (\Y_m| \B_m , \A_m ) = \frac{H_q(\X_n;\Y_m) H_q(\A_{n-1};\B_m)}{ H_q(\X_n;\B_m) H_q(\A_{n-1};\Y_m) }.
    \end{equation}
\end{corollary}

\subsection{Spin Hall-Littlewood Rational Functions}

Here we define a family of symmetric functions known in literature as \emph{inhomogeneous Spin Hall-Littlewood rational function} \cite{Borodin2014vertex,BorodinPetrov2016inhom}. They are important as they are self-orthogonal with respect to the scalar product \eqref{eq:sHL scalar product} defined below and satisfy dual Cauchy identities when paired with the spin $q$-Whittaker polynomials.

\begin{definition}
    Define the spin Hall Littlewood rational functions 
    \begin{equation} \label{eq:sHL symmetrized sum}
        \shlF_\lambda(\U_m|\A,\B) = \frac{(1-q)^m}{(q;q)_{m-\ell(\lambda)}} \psi_{\lambda'}(\A,\B) \sum_{\sigma \in S_m} \sigma \left( \prod_{1\le \alpha < \beta \le m} \frac{u_\alpha - q u_\beta}{u_\alpha - u_\beta} \prod_{i=1}^{\ell(\lambda)} \phi_{\lambda_i} (u_i|\A,\B) \right)
    \end{equation}
    with
    \begin{equation}
        \phi_k (u|\A,\B) = (-1)^k \frac{u}{1-a_k u} \prod_{j=1}^{k-1} \frac{ u - b_j }{1-a_j u}.
    \end{equation}
\end{definition}

Let us compare the definition \eqref{eq:sHL symmetrized sum} with others that have appeared in literature.

\begin{remark} \label{rem:matching notation shl}
    The functions $\shlF$ correspond to those in \cite{BorodinPetrov2016inhom} (denoted by $\mathsf{F}( \, \cdot \, | \Xi , \mathsf{S})$) as
    \begin{equation}
        \shlF_\lambda(\U_m|\A,\B)\big|_{a_j = s_j\xi_j, \, b_j=s_j/\xi_j} = \frac{ (-1)^{|\lambda|} }{(q;q)_{m-\ell(\lambda)} } \prod_{r\ge 1} \frac{(s_r^2;q)_{\lambda_r'-\lambda_{r+1}'} }{(q;q)_{\lambda_r'-\lambda_{r+1}'} } \left( \prod_{i\ge1} \frac{1}{\xi_i^{\lambda_{i+1}'} } \right)\mathsf{F}_\lambda(\U_m|\Xi,\mathsf{S})\big|_{s_0=0,\, \xi_0=1}.
    \end{equation}
    They can also be matched with those in \cite{Borodin_Korotkikh_Inhom} (denoted by $\widetilde{\mathsf{F}}$) as
    \begin{equation}
        \shlF_\lambda(\U_m|\A,\B)\big|_{a_i = s_i\xi_i, \, b_i=s_i/\xi_i} = (-1)^{|\lambda|} \left( \prod_{r\ge 1} \frac{1}{ \xi_r^{\lambda_{r+1}'}} \frac{(s_r^2;q)_{\lambda_r'-\lambda_{r+1}'} }{(q;q)_{\lambda_r'-\lambda_{r+1}'} }   \right) \widetilde{\mathsf{F}}_\lambda (\U_m|\Xi, \mathsf{S}). 
    \end{equation}
\end{remark}

\begin{remark}
    The functions $\shlF$ reduce to the Hall-Littlwood polynomials as
    \begin{equation}
        P_\lambda(\U_{m};0,q) =  (-1)^{|\lambda|} \shlF_\lambda( \U_m|\A,\B)\big|_{a_1=b_1=a_2=b_2\cdots=0}.
    \end{equation}
\end{remark}

\begin{definition}
    For rational functions $f,g:\mathbb{T}^m \to \mathbb{C}$, define the scalar product
    \begin{equation} \label{eq:sHL scalar product}
        \langle f(\Z_m) | g(\Z_m) \rangle_m' = \frac{1}{m!} \oint_{\mathbb{T}^m} \frac{\diff z_1}{2 \pi \mathrm{i} z_1 } \cdots \frac{\diff z_m}{2 \pi \mathrm{i} z_m } \prod_{1 \le \alpha \neq \beta \le m} \frac{z_\alpha - z_\beta}{z_\alpha - q z_\beta} f(\Z_m) g(\Zbar_m).
    \end{equation}
\end{definition}

\begin{proposition}[\cite{BorodinPetrov2016inhom}, Corollary 7.5] \label{prop:orthogonality shl}
    The functions $\shlF_\lambda$ satisfy the orthogonality
        \begin{equation} \label{eq:orthogonality shl}
            \langle \shlF_\lambda(\Z_m|\A,\B) | \shlF_\mu(\Z_m|\B,\A) \rangle_m' 
                = \mathbf{1}_{\lambda=\mu} \frac{ (1-q)^m}{(q;q)_{m-\ell(\lambda)}} \psi_{\lambda'}(\A,\B),
        \end{equation}
        where $\psi$ was defined in \eqref{eq:psi}.
\end{proposition}

For the next proposition define 
\begin{equation} \label{eq:E}
    E(\X_n;\Y_m) = \prod_{i=1}^n \prod_{j=1}^m(1+ x_i y_j ).
\end{equation}

\begin{proposition}[\cite{Borodin_Korotkikh_Inhom}, Corollary 5.12]
     Fix $m,n\in \mathbb{N}$. We have
     \begin{equation} 
        \sum_{\lambda \in \mathrm{Box}(m,n)}  \shlF_{\lambda'} (\U_m|\A,\B)  \sqwFBK_\lambda(\X_n|\A,\B) = \frac{E(\X_n;-\U_m)}{E(\A_{n};-\U_m)},
    \end{equation}
    where $-\U_m=(-u_1,\dots , -u_m)$ and the $\mathrm{Box}(m,n)$ was defined in \eqref{eq:box}.
\end{proposition}

\begin{corollary}
     Fix $m,n\in \mathbb{N}$. We have
     \begin{equation} \label{eq:dual CI}
        \sum_{\lambda \in \mathrm{Box}(m,n) }  \shlF_{\lambda'} (\U_m|\A,\B)  \sqwF_\lambda(\X_n|\A_{n-1},\B_{n-1}) = \frac{E(\X_n;-\U_m)}{E(\A_{n-1};-\U_m)}.
    \end{equation}
\end{corollary}

\begin{definition}
    Let $\lambda \in \mathbb{Y}_{n+k}$ and $\mu \in \mathbb{Y}_n$. Define the skew spin Hall-Littlewood symmetric function as 
    \begin{equation}
        \shlF_{\lambda/\mu}(\U_k|\A,\B) = \frac{\left\langle \shlF_\lambda(\U_k,\Z_n | \A, \B) \middle| \shlF_\mu(\Z_n | \A, \B) \right\rangle_n'}{ \left\langle \shlF_\mu (\Z_n | \A, \B) | \shlF_\mu (\Z_n | \A, \B) \right\rangle_n' },
    \end{equation}
    where $\shlF_{\lambda/\mu}(\U_k, \Z_n|\A,\B)$ denotes the spin Hall-Littlewood rational function \eqref{eq:sHL symmetrized sum} in $k+n$ variables $(u_1,\dots, u_k,  z_1, \dots, z_n)$.
\end{definition}

\begin{proposition}[\cite{Borodin_Korotkikh_Inhom}, Corollary 5.11]
    The spin $q$-Whittaker polynomials satisfy the Pieri rule
    \begin{equation} \label{eq:pieri rule}
        \frac{E(\X_n;-u)}{E(\A_n;-u)} \sqwFBK_\mu(\X_n|\A_n,\B_n)=\sum_{\lambda} \shlF_{\lambda'/\mu'} (u|\A,\B) \sqwFBK_\lambda(\X_n|\A_n,\B_n).
    \end{equation}
\end{proposition}

\section{Eigenrelations} \label{sec:eigenrelations}

The main results of this section are \Cref{thm:eigenrelation sqW lambda n} and \Cref{thm:eigenrelation sqW lambda 1}, which give two eigenrelations for the spin $q$-Whittaker polynomials $\sqwF$. They represent fully inhomogeneous generalizations respectively of \cite[Theorem 3.9]{MucciconiPetrov2020} and \cite[Theorem 3.10]{MucciconiPetrov2020}. The proofs of these theorems will only be sketched, since the arguments can be adapted from those in \cite{MucciconiPetrov2020}.

\subsection{Eigenrelations involving $\lambda_n$}
For a set of $n$ variables $\X_n=(x_1,\dots,x_n)$ and for any $i\in \{1,\dots, n\}$, we denote by $T_{q,x_i}$ be the $q$-shift operator that acts on $n$-variate functions $f(\X_n)$ as
\begin{equation}
    T_{q,x_i} [f (\X_n)] = f(x_1,\dots,x_{i-1},q x_i , x_{i+1}, \dots, x_n). 
\end{equation}
We introduce the operator
\begin{equation}
    \mathfrak{D}_{\B_{n-1}} = \sum_{i=1}^n \prod_{r=1}^{n-1} (1-x_i b_r) \prod_{\substack{ j=1 \\ j\neq i}}^N \frac{1}{1-x_i/x_j} T_{q,x_i}.
\end{equation}
More compactly, one can express the action of $\mathfrak{D}_{\B_{n-1}}$ on an $n$-variate function $f(\X_n)$ as
\begin{equation} \label{eq:D B conjugation}
    \mathfrak{D}_{\B_{n-1}}[f (\X_n)]   = H_q(\X_n;\B_{n-1})^{-1} D_1 \left[ H_q( \X_n;\B_{n-1}) f (\X_n) \right] ,
\end{equation}
where $D_1$ is the $q$-Whittaker $q$-difference operator
\begin{equation}
    D_1 = \sum_{i=1}^n  \prod_{\substack{ j=1 \\ j\neq i}}^n \frac{1}{1-x_i/x_j} T_{q,x_i}.
\end{equation}

\begin{theorem} \label{thm:eigenrelation sqW lambda n}
    We have
        \begin{equation} \label{eq:eigenrelation sqW lambda n}
            \mathfrak{D}_{\B_{n-1}} \mathbb{F}_\lambda(\X_n|\A_{n-1}, \B_{n-1}) = q^{\lambda_N} \mathbb{F}_\lambda (\X_n|\A_{n-1}, \B_{n-1}).
        \end{equation}
\end{theorem}

\begin{remark}
    In case $\B_{n-1} = (0,\dots,0)$, \Cref{thm:eigenrelation sqW lambda n} reduces to the known eigenrelation for $q$-Whittaker polynomials \cite[VI]{Macdonald1995}
    \begin{equation}
        D_1 P_\lambda (\X_n;q,0) = q^{\lambda_n} P_\lambda(\X_n;q,0).
    \end{equation}
    On the other hand, setting $\B_{n-1} = \A_{n-1} =(-s,\dots,-s)$, \Cref{thm:eigenrelation sqW lambda n} becomes \cite[Theorem 3.9]{MucciconiPetrov2020}.
\end{remark}

The proof of \Cref{thm:eigenrelation sqW lambda n} uses an eigenrelation for the spin Hall-Littlewood rational functions $\shlF$, which is dual to \eqref{eq:eigenrelation sqW lambda n} and that we state below. 

\medskip

Denote by $V(m)$ the space of symmetric rational functions of $m$ variables of degree at most $1$ in each variable. Alternatively, $V(m)$ is the space of rational functions that can be written as $a(\U_m)/b(\U_m)$ where $a,b$ are symmetric polynomials and $\deg_{u_i}(a) - \deg_{u_i}(b) \le 1$ for all $i=1,\dots,m$. Define the operator $\mathfrak{D}^*_{n}$, which acts on function of $V(m)$ as
\begin{equation}
    \mathfrak{D}^*_{n} = \sum_{j=1}^m \prod_{\substack{ l=1 \\ l\neq j }}^m \frac{u_j - q u_l}{u_j-u_l} \mathfrak{C}_{j,n},
\end{equation}
where
\begin{equation}
    \mathfrak{C}_{j,n} = \prod_{i=1}^{n-1} (-a_i) \frac{ u_j - b_i }{1-a_i u_j} \lim_{\varepsilon \to 0} \varepsilon T_{\varepsilon^{-1},u_j}.
\end{equation}

\begin{proposition} \label{prop:eigenrelation shl lambda prime N}
    Let $\lambda \in \mathrm{Box}(n,m)$. Then, we have
    \begin{equation} \label{eq:eigenrelation shl lambda prime N}
        \mathfrak{D}^*_{n} \shlF_\lambda (\U_m|\A,\B)\big|_{a_n=0} = \frac{1-q^{\lambda_n'}}{1-q} \shlF_\lambda (\U_m|\A,\B)\big|_{a_n=0}.
    \end{equation}    
\end{proposition}

\begin{proof}[Sketch of proof]
    The eigenrelation \eqref{eq:eigenrelation shl lambda prime N} is the inhomogeneous variant of \cite[Theorem 3.4]{MucciconiPetrov2020} and its proof consists of a direct calculation of the action of the operator $\mathfrak{D}^*_{n}$ on the symmetrized sum \eqref{eq:sHL symmetrized sum}. The operator $\mathfrak{C}_{i,n}$ has the property that
    \begin{equation}
        \mathfrak{C}_{i,n} \left[ \phi_k( u_i|\A,\B)\big|_{a_n=0} \right] = \begin{cases} 
            \phi_k(u_i|\A,\B)\big|_{a_n=0} \qquad & \text{if } k=n,
            \\
            0 \qquad & \text{otherwise},
        \end{cases} 
    \end{equation}
    for all $i =1,\dots, m$. As a result, calling
    \begin{equation}
        \Phi_\lambda(\U_m) = \prod_{j=1}^m \phi_{\lambda_j}(u_{j}|\A,\B)\big|_{a_n=0},
    \end{equation}
    for any $\sigma \in S_m$, we have
    \begin{equation}
        \mathfrak{C}_{i,n} \left[ \sigma\left( \Phi_\lambda(\U_m) \right) \right] = \begin{cases} 
            \sigma\left( \Phi_\lambda(\U_m) \right) \qquad & \text{if } i\in \sigma \{1,\dots, \lambda'_n \},
            \\
            0 \qquad & \text{otherwise}.
        \end{cases} 
    \end{equation}
    One can then evaluate
    \begin{equation}
        \begin{split}
            &\mathfrak{D}_n^*\shlF_\lambda(\U_m|\A,\B) \big|_{a_n=0}
            \\
            &= C(\lambda) \sum_{j=1}^m \sum_{\substack{ \sigma \in S_n \\ i \in \{1, \dots , \lambda_N' \} } } \prod_{ \substack{ l=1 \\ l\neq j} } \frac{u_j -q u_l}{u_j - u_l} \lim_{\varepsilon\to 0} \sigma\left\{ \prod_{1\le \alpha < \beta \le m} \frac{u_{\alpha} - q u_{\beta}}{u_{\alpha} - u_{\beta}} \right\}\bigg|_{u_i=\varepsilon}   \mathfrak{C}_{i,N} \left[ \sigma \left(\Phi_\lambda(\U_m) \right) \right],
        \end{split}
    \end{equation}
    with the constant prefactor set as $C(\lambda) = \frac{(1-q)^n}{(q;q)_{n-\ell(\lambda)}} \psi_{\lambda'}(\A,\B)\big|_{a_n=0}$.
    Following \cite[Theorem 3.4]{MucciconiPetrov2020}, the limit in the right hand side is responsible for the creation of factors $q^{\bar{k}}$ with $\bar{k}=\sigma^{-1}(i)$ and the double summation can be rearranged as
    \begin{equation}
        C(\lambda) \sum_{\bar{k}=1}^{\lambda_n'} q^{\bar{k}} \sum_{\tau \in S_m} \tau \left( \prod_{1\le \alpha < \beta \le m} \frac{u_\alpha - q u_\beta}{u_\alpha - u_\beta} \Phi_\lambda(\U_m) \right),
    \end{equation}
    which equals the right hand side of \eqref{eq:eigenrelation shl lambda prime N}.
\end{proof}

\begin{proof}[Sketch of proof of \Cref{thm:eigenrelation sqW lambda n}]
    Armed with \Cref{prop:eigenrelation shl lambda prime N} we prove that
    \begin{equation} \label{eq:equality actions on partition functions}
        \left( 1 - (1-q)\mathfrak{D}^*_n \right) \left[ \frac{E(\X_n; -\U_m)}{E(\A_{n-1}; -\U_m)} \right] = \mathfrak{D}_{\B_{n-1}} \left[ \frac{E(\X_n; -\U_m)}{E(\A_{n-1}; -\U_m)} \right],
    \end{equation}
    which implies, by the dual Cauchy identity \eqref{eq:dual CI}
    \begin{equation}
        \begin{split}
            &\sum_\lambda  q^{\lambda_n} \shlF_{\lambda'} (\U_m|\A,\B) \big|_{a_n=0}  \sqwF_\lambda(\X_n|\A_{n-1},\B_{n-1}) 
            \\
            & = \sum_\lambda  q^{\lambda_n} \shlF_{\lambda'} (\U_m|\A,\B) \big|_{a_n=0} \mathfrak{D}_{\B_{n-1}}\left[ \sqwF_\lambda(\X_n|\A_{n-1},\B_{n-1}) \right]
        \end{split}
    \end{equation}
    and by orthogonality \eqref{eq:orthogonality shl} and the eigenrelation \eqref{eq:eigenrelation sqW lambda n}. Checking the identity \eqref{eq:equality actions on partition functions} is a direct calculation which uses the integral representation of the action of the $q$-difference operator $\mathfrak{D}_{\B_{n-1}}$ on fully factorized functions
    \begin{equation} \label{eq:integral rep D b}
        \frac{\mathfrak{D}_{\B_{n-1}}[h(x_1) \cdots h(x_n)]}{h(x_1) \cdots h(x_n)} = - \oint_{x_1,\dots,x_n}  \frac{\prod_{r=1}^{n-1}(1-b_r z)}{\prod_{i=1}^n (1-z/x_i)} \frac{h(qz)}{h(z)} \frac{\diff z}{2 \pi \mathrm{i} z}.
    \end{equation}
    Above, the integration contour only encircles poles $x_1,\dots,x_n$. An analogous calculation was reported in \cite[Theorem 3.9]{MucciconiPetrov2020} and therefore we omit the details.
\end{proof}

\subsection{Eigenrelations involving $\lambda_1$}

Introduce the $q$-difference operator
\begin{equation}
        \overline{\mathfrak{D}}_{\A_{n-1}}  = \sum_{i=1}^n \prod_{r=1}^{n-1} (1-a_r/x_i) \prod_{\substack{ j=1 \\ j\neq i}}^n \frac{1}{1-x_j/x_i} T_{q^{-1},x_i}.
\end{equation}
The action of $\overline{\mathfrak{D}}_{\A_{n-1}}$ on an $n$-variate function $f(\X_n)$ can also be expressed as
\begin{equation} \label{eq:D bar by conjugation}
        \overline{\mathfrak{D}}_{\A_{n-1}} [f (\X_n)] = H_q(\Xbar_n;\A_{n-1})^{-1} \overline{D}_1 \left[ H_q( \Xbar_n ;\A_{n-1}) f(\X_n) \right],            
\end{equation}
where 
\begin{equation}
    \overline{D}_1 = \sum_{i=1}^n  \prod_{\substack{ j=1 \\ j\neq i}}^n \frac{1}{1-x_j/x_i} T_{q^{-1},x_i}.
\end{equation}

\begin{theorem} \label{thm:eigenrelation sqW lambda 1}
    We have
    \begin{equation}\label{eq:eigenrelation sqW lambda 1}
        \overline{\mathfrak{D}}_{\A_{n-1}} \overline{\mathbb{F}}_\lambda(\X_n|\A_{n-1}, \B_{n-1}) = q^{-\lambda_1} \overline{\mathbb{F}}_\lambda (\X_n|\A_{n-1}, \B_{n-1}).
    \end{equation}
\end{theorem}

The proof of \Cref{thm:eigenrelation sqW lambda 1} follows the same blueprint of \Cref{thm:eigenrelation sqW lambda n} and uses another eigenrelation for the spin Hall-Littlewood rational functions $\shlF$. To state it we define the operator
\begin{equation}
    \overline{\mathfrak{D}}_r^* = \sum_{ \substack{ I \subset \{ 1,\dots,n\} \\ |I|=r } } \prod_{ \substack{ i \in I \\ j \in \{ 1,\dots,n \} \setminus I } }^M \frac{u_j - q u_i}{u_j - u_i} T_{0,I},
\end{equation}
where
\begin{equation}
    T_{0,I} = \prod_{i \in I} T_{0,i}.
\end{equation}

\begin{proposition} \label{prop:eigenrelation sHL lambda 1}
    We have
    \begin{equation} \label{eq:eigenrelation sHL lambda 1}
        \overline{\mathfrak{D}}_r^* \shlF_\lambda = e_r(1,q,\dots,q^{M-\lambda_{1}'-1}) \shlF_\lambda.
    \end{equation}
\end{proposition}

\begin{proof}[Sketch of proof]
    \Cref{prop:eigenrelation sHL lambda 1} is the inhomogeneous variant of \cite[Theorem 8.2]{BufetovMucciconiPetrov2018}, where parameters $\A,\B$ were set as $a_1,\dots,a_{n-1}=b_1=\cdots=b_{n-1}=-s$. It's proof is a direct evaluation of the action of $\overline{\mathfrak{D}}_r^*$ on the explicit expression \eqref{eq:sHL symmetrized sum} and since the argument are analogous to \cite[Theorem 8.2]{BufetovMucciconiPetrov2018} we will not repeat it here.
\end{proof}

\begin{proof}[Sketch of proof of \Cref{thm:eigenrelation sqW lambda 1}]
    Define the operator
    \begin{equation}
        \widetilde{\mathfrak{D}} = q^{-n} \left( 1 + (q-1) \overline{\mathfrak{D}}^*_1 \right)
    \end{equation}
    and notice, by \eqref{prop:eigenrelation sHL lambda 1}, that $\widetilde{\mathfrak{D}}$ acts on spin Hall-Littlewood rational functions $\shlF_\lambda$ diagonally with eigenvalue $q^{-\lambda_1'}$.
    We prove that
    \begin{equation} \label{eq:duality}
         \widetilde{\mathfrak{D}} \left[ \frac{E(\X_n;-\U_m)}{E(\A_{n-1};-\U_m)} \right] = \overline{\mathfrak{D}}_{\A_{n-1}} \left[ \frac{E(\X_n;-\U_m)}{E(\A_{n-1};-\U_m)} \right].
    \end{equation}
    This identity proves \eqref{eq:eigenrelation sqW lambda 1} thanks to the orthogonality of the spin Hall-Littlewood functions \eqref{eq:orthogonality shl}.
    The action of $\widetilde{\mathfrak{D}}$ on a fully factorized function $g(u_1)\cdots g(u_m)$ with $g(0)=1$ can be written as
    \begin{equation}
        \frac{\widetilde{\mathfrak{D}}[g(u_1) \cdots g(u_m)] }{ g(u_1) \cdots g(u_m) } = 
        \oint_{\gamma_{0,\boldsymbol{u}}} \prod_{j=1}^n \frac{w-q^{-1}u_j}{w-u_j} \frac{1}{g(w)} \frac{\diff w}{ 2 \pi \mathrm{i} w},
    \end{equation}
    where the contour is positively oriented and contains $0$, $u_i$ for $i=1,\dots,n$ and no other singularity of the integrand. On the other hand the action of the operator $\overline{\mathfrak{D}}_{\A_{n-1}}$ on a factorized function is
    \begin{equation} \label{eq:integral rep Dbar a}
        \frac{\overline{\mathfrak{D}}_{\A_{n-1}} [h(x_1) \cdots h(x_n)]}{ h(x_1) \cdots h(x_n) } = \frac{1}{2\pi \mathrm{i}} \oint_{ x_1,\dots, x_n } \frac{\prod_{i=1}^{l-1} (z+a_i) }{\prod_{i=1}^l (z-x_i)} \frac{h(z/q)}{h(z)} \diff z.
    \end{equation}
    Matching the left and right hand side of \eqref{eq:dual CI kernel} reduces then to a simple manipulation of the above integral expressions; see \cite[Theorem 8.2]{BufetovMucciconiPetrov2018} for details.
\end{proof}

\begin{remark}
    A consequence of \Cref{thm:orthogonality inhom sqW}, \Cref{thm:sqW are a basis} below is that the $q$-difference operators $\mathfrak{D}_{\B_{n-1}}$ and $\overline{\mathfrak{D}}_{\A_{n-1}}$ commute with each other. This fact is a priori rather nontrivial, although it can be verified directly for instance through the integral representation of the action of these operators \eqref{eq:integral rep D b}, \eqref{eq:integral rep Dbar a}; see \cite[Proposition 3.11]{MucciconiPetrov2020} for a similar calculation.
\end{remark}

\section{Proofs of main theorems} \label{sec:proofs}

\subsection{Torus scalar product for $q$-Whittaker polynomials}

For any two functions $f,g: \mathbb{T}^n \to \mathbb{C}$, we introduce the scalar product
\begin{equation}
    \langle f \, | \,g \rangle_n = \oint_{\mathbb{T}^n} f(\Z_n) g(\Zbar_n) \Delta(\Z_n) \mathrm{d} z_1 \cdots \mathrm{d} z_n,
\end{equation}
where the density $\Delta(\Z_n)$ was defined in \eqref{eq:Delta}. The scalar product $\langle\cdot | \cdot \rangle_n$ is a the $t=0$ case of Macdonald's torus scalar product; see \cite[VI.9]{Macdonald1995}. The $q$-Whittaker polynomials $P$ satisfy the orthogonality
\begin{equation} \label{eq:orthogonality qW}
    \left\langle P_\lambda \, \middle| \, P_\mu \right\rangle_n = \mathbf{1}_{\lambda=\mu} \mathsf{c}_n(\lambda) ,
    \qquad
    \text{where}
    \qquad
    \mathsf{c}_n(\lambda ) = \prod_{k=1}^{n-1} \frac{(q;q)_{\lambda_k-\lambda_{k+1}}}{(q;q)_\infty}.
\end{equation}
It is clear, using the definition \eqref{eq:sqW scalar product} of the scalar product $\llangle | \rrangle$ and the density $\Delta(\Z_n|\A_{n-1} , \B_{n-1})$ from \eqref{eq: delta A B}, that
\begin{equation} \label{eq:relation two scalar products}
    \llangle f \, | \, g \rrangle_n = \frac{\prod_{k=1}^{n-1} (a_k b_k;q)_\infty}{H_q(\A_{n-1}; \B_{n-1})} \left\langle f(\Z_n) H_q( \Z_n; \B_{n-1} ) \, \middle| \, g(\Z_n)  H_q( \Z_n; \A_{n-1} ) \right\rangle_n.
\end{equation}

\subsection{Proof of \Cref{thm:orthogonality inhom sqW}} 

The proof of \Cref{thm:orthogonality inhom sqW} uses a number of preliminary results, which we organize in several lemmas.

\begin{lemma} \label{lem:evaluation scalar product kernels}
    Fix integers $0<m<n$ and $k>0$. Then, we have
    \begin{equation} \label{eq:scalar product factorized quantities}
        \left\llangle \frac{H_q(\Z_n; \Y_m) H_q(\A_{n-1};\B_m) }{H_q(\Z_n; \B_{m}) H_q(\A_{n-1};\Y_m)} \, \middle| \, \frac{E (\Z_n; -\U_k) }{E(\B_{n-1}; -\U_k)}\right\rrangle_n = \frac{\prod_{k=1}^{n-1} (a_k b_k;q)_\infty}{(q;q)_\infty^{n-1}} \frac{  E(\Y_m; -\U_k) }{ E(\B_{m}; -\U_k) },
    \end{equation}
    where we recall that functions $H_q,E$ were defined in \eqref{eq:H}, \eqref{eq:E} respectively.
\end{lemma}
\begin{proof}
    The left hand side of \eqref{eq:scalar product factorized quantities} can be simplified as
    \begin{equation}
        \begin{split}
            &\frac{H_q(\A_{n-1};\B_m)}{H_q(\A_{n-1};\Y_m) E(\B_{n-1}; -\U_k) } \left\llangle \frac{H_q(\Z_n; \Y_m)}{H_q(\Z_n; \B_{m})} \, \middle| \, E (\Z_n; -\U_k)\right\rrangle_n
            \\
            & = \frac{H_q(\A_{n-1};\B_{m}) \prod_{k=1}^{n-1} (a_k b_k;q)_\infty }{H_q(\A_{n-1};\Y_m) E(\B_{n-1} ; -\U_k) H_q( \A_{n-1}; \B_{n-1}) } 
            \\
            & \qquad\times \left\langle \frac{H_q(\Z_n; \Y_m)}{H_q(\Z_n; \B_{m})} H_q( \Z_n; \B_{n-1}) \, \middle| \, E (\Z_n; -\U_k) H_q(\Z_n; \A_{n-1})  \right\rangle_n
            \\
            & = \frac{\prod_{k=1}^{n-1} (a_k b_k;q)_\infty}{H_q(\A_{n-1};\Y_m) H_q(\A_{n-1}; \tau^m\B_{n-1} ) E(\B_{n-1}; -\U_k) } 
            \\
            &\qquad\times \left\langle H_q(\Z_n; \Y_m, \tau^m\B_{n-1} ) \, \middle| \, E (\Z_n; -\U_k) H_q(\Z_n; \A_{n-1})  \right\rangle_n
        \end{split}
    \end{equation}
    where $\tau^m\B_{n-1}=(b_{m+1},\dots,b_{n-1})$ as in \eqref{eq:tau k}. The factorized expressions in the scalar product in the right hand side can be expanded as
    \begin{equation}
        \begin{aligned}
            H_q(\Z_n; \Y_m, \tau^m\B_{n-1} ) = \sum_{\lambda} P_\lambda(\Z_n;q,0) Q_\lambda(\Y_m, \tau^m\B_{n-1} ;q,0 )
            \\
            E (\Z_n; -\U_k) H_q(\Z_n; \A_{n-1}) = \sum_{\mu} P_\mu(\Z_n;q,0) Q_\mu(\A_{n-1},-\widetilde{\U_k};q,0).
        \end{aligned}
    \end{equation}
    Here, in the first expression, the term $Q_\lambda(\Y_m, \tau^m\B_{n-1} ;q,0)$ is the $q$-Whittaker polynomial $Q_\lambda$ with $n-1$ variables specialized as $(y_1,\dots,y_m,b_{m+1},\dots,b_{n-1})$, whereas in the second line one can define the $Q_\mu$ term through the orthogonality \eqref{eq:orthogonality qW} as
    \begin{equation}
            Q_\mu(\A_{n-1},-\widetilde{\U_k};q,0) = \frac{1}{\mathsf{c}_n(\lambda)} \left \langle E (\Z_n; -\U_k) H_q(\Z_n; \A_{n-1}) \middle| P_\mu(\Z_n;q,0) \right \rangle_n.
    \end{equation}
    Using the above expansions we find
    \begin{equation}
        \begin{aligned}
            & \left\langle H_q(\Z_n; \Y_m, \tau^m \B_{n-1}) \, \middle| \, H_q(\Z_n; \A_{n-1},-\widetilde{\U_k})  \right\rangle_n
            \\
            &
            = \sum_{\lambda ,\mu} \left\langle P_\lambda(\Z_n;q,0) \, \middle| \, P_\mu(\Z_n;q,0) \right\rangle_n Q_\lambda( \Y_m, \tau^m \B_{n-1} ;q,0 ) Q_\mu( \A_{n-1},-\widetilde{\U_k} ;q,0 )
            \\
            &
            = \sum_{\lambda ,\mu} \mathsf{c}_n(\lambda) Q_\lambda ( \Y_m, \tau^m\B_{n-1} ;q,0 ) Q_\mu( \A_{n-1},-\widetilde{\U_k} ; q,0 )
            \\
            &
            = \frac{1}{(q;q)_\infty^{n-1}} \sum_{\lambda}  P_\lambda ( \Y_m, \tau^m\B_{n-1} ;q,0 ) Q_\mu( \A_{n-1}, -\widetilde{\U_k} ;q,0 )
            \\
            & = \frac{1}{(q;q)_\infty^{n-1}}
            H_q(\Y_m; \A_{n-1}) H_q( \tau^m \B_{n-1}; \A_{n-1}) E(\Y_m; \U_k) E(\tau^m\B_{n-1}; -\U_k),
        \end{aligned}
    \end{equation}
    where we used the fact that
    \begin{equation}
        \mathsf{c}_n(\lambda) Q_\lambda(\Y_m,\tau^m\B_{n-1} ;q,0 ) = \frac{1}{(q;q)_\infty^{n-1}} P_\lambda (\Y_m, \tau^m\B_{n-1} ;q,0 ),
        \qquad \text{for all $\lambda$ such that } \lambda_n=0.
    \end{equation}
    One can see that the previous identity hold directly from \eqref{eq:sqW at q=0}.
    Then, we have
     \begin{equation}
        \begin{aligned}
            &\frac{H_q(\A_{n-1};\B_m)}{H_q(\A_{n-1};\Y_m) E(\B_{n-1}; -\U_k) } \left\llangle \frac{H_q(\Z_n; \Y_m)}{H_q(\Z_n; \B_{m})} \, \middle| \, E (\Z_n; -\U_k)\right\rrangle_n
            \\
            &
            = \frac{\prod_{k=1}^{n-1} (a_k b_k;q)_\infty}{(q;q)_\infty^{n-1}} \frac{H_q(\Y_m; \A_{n-1}) H_q(\tau^m\B_{n-1}; \A_{n-1}) E(\Y_m; -\U_k) E(\tau^m\B_{n-1} ; -\U_k)}{H_q(\A_{n-1};\Y_m) H_q(\A_{n-1};\tau^m\B_{n-1}) E(\B_{n-1};-\U_k) }
            \\
            &
            = \frac{\prod_{k=1}^{n-1} (a_k b_k;q)_\infty}{(q;q)_\infty^{n-1}} \frac{  E(\Y_m; -\U_k) }{ E(\B_{m};-\U_k) },
        \end{aligned}
    \end{equation}
    which completes the proof.
\end{proof}

\begin{lemma} \label{lem:integral rep sqW}
    Fix integers $0<m<n$. Then, for every $\lambda \in \mathbb{Y}_m$, we have
    \begin{equation} \label{eq:expansion sqW scalar prod}
        \sqwFBK_\lambda(\Y_m|\B_m, \A_m) = \prod_{k=1}^{n-1} \frac{(q;q)_\infty}{(a_k b_k;q)_\infty}  \left\llangle \frac{H_q(\Z_n; \Y_m) H_q(\A_{n-1};\B_m) }{H_q(\Z_n; \B_{m}) H_q(\A_{n-1};\Y_m)} \, \middle| \, \sqwF_\lambda(\Z_n | \B_{n-1}, \A_{n-1} ) \right\rrangle_n.
    \end{equation}
    Moreover, the following expansion holds
    \begin{equation} \label{eq:expansion sqW infinite sum}
        \begin{split}
            &\sqwFBK_\lambda(\Y_m|\B_m, \A_m) 
            \\
            &= \prod_{k=1}^{n-1} \frac{(q;q)_\infty}{(a_k b_k;q)_\infty} \sum_{\mu}  \left\llangle \sqwF_\mu(\Z_n| \A_{n-1},\B_{n-1})  \, \middle| \, \sqwF_\lambda(\Z_n | \B_{n-1}, \A_{n-1} ) \right\rrangle_n \psi_\mu(\A,\B) \sqwFBK_\mu(\Y_m | \B_m, \A_m ).
        \end{split}
    \end{equation}
\end{lemma}

\begin{proof}
    First notice that the expansion \eqref{eq:expansion sqW infinite sum} follows directly from \eqref{eq:expansion sqW scalar prod}, employing the Cauchy identity \eqref{eq:inhom CI sqW sqW} to expand the expression in the left term of the scalar product. Therefore, we only need to prove \eqref{eq:expansion sqW scalar prod}. By the dual Cauchy identity \eqref{eq:dual CI} and the orthogonality of the spin Hall-Littewood rational functions \eqref{eq:orthogonality shl}, it suffices to prove that the summation
    \begin{equation} \label{eq:summation sHL scalar prod}
        \sum_{\lambda \in \mathrm{Box}(k,n) } \shlF_{\lambda'} (\U_k | \B, \A) \prod_{k=1}^{n-1} \frac{(q;q)_\infty}{(a_k b_k;q)_\infty}  \left\llangle \frac{H_q(\Z_n; \Y_m) H_q(\A_{n-1};\B_m) }{H_q(\Z_n; \B_{m}) H_q(\A_{n-1};\Y_m)} \,  \middle| \, \sqwF_\lambda(\Z_n | \B_{n-1}, \A_{n-1} ) \right\rrangle_n 
    \end{equation}
    equals
    \begin{equation} \label{eq:dual CI kernel}
        \frac{E(\Y_m; -\U_k)}{E(\B_{m}; -\U_k)}.
    \end{equation}
    Bringing the summation inside the scalar product and using the the dual Cauchy identity \eqref{eq:dual CI}, the summation \eqref{eq:summation sHL scalar prod} becomes
    \begin{equation}
        \prod_{k=1}^{n-1} \frac{(q;q)_\infty}{(a_k b_k;q)_\infty} \left\llangle \frac{H_q(\Z_n; \Y_m) H_q(\A_{n-1};\B_m) }{H_q(\Z_n; \B_{m}) H_q(\A_{n-1};\Y_m)} \, \middle| \, \frac{E (\Z_n; -\U_k) }{E(\B_{n-1}; -\U_k)}\right\rrangle_n,
    \end{equation}
    which, by \Cref{lem:evaluation scalar product kernels}, equals \eqref{eq:dual CI kernel}, completing the proof of \eqref{eq:expansion sqW scalar prod}. 
\end{proof}

\begin{lemma} \label{lem:adjoint}
    The operators $\overline{\mathfrak{D}}_{\A_{n-1}}$ and $\overline{\mathfrak{D}}_{\B_{n-1}}$ are adjoint with respect to the scalar product $\llangle \cdot \, | \, \cdot \rrangle_n$. Similarly, operators $\mathfrak{D}_{\A_{n-1}}, \mathfrak{D}_{\B_{n-1}}$ are adjoint with respect to the same scalar product.
\end{lemma}
\begin{proof}
    This is the result of a simple computation which we report for completeness. Using \eqref{eq:D bar by conjugation} and \eqref{eq:relation two scalar products}, we have
    \begin{equation}
        \begin{split}
            & \left \llangle \overline{\mathfrak{D}}_{\A_{n-1}} \left[ f(\Z_n) \right] \, \middle| \, g(\Z_n) \right \rrangle_n 
            \\
            &= \left\llangle H_q(\Zbar_n;\A_{n-1})^{-1} \overline{D}_1 \left[ H_q(\Zbar_n;\A_{n-1}) f(\Z_n) \right]  \, \middle| \, g(\Z_n) \right\rrangle_n
            \\
            & = \frac{1}{H_q(\A_{n-1},\B_{n-1})}
            \\
            & \qquad \times \left\langle H_q( \Zbar_n;\A_{n-1})^{-1} \overline{D}_1 \left[ H_q(\Zbar_n;\A_{n-1}) f(\Z_n) \right] H_q(\Z_n; \B_{n-1}) \, \middle| \, g(\Z_n) H_q(\Z_n; \A_{n-1}) \right\rangle_n
            \\
            & = \frac{1}{H_q(\A_{n-1},\B_{n-1})}  \left\langle  \overline{D}_1 \left[ H_q(\Zbar_n;\A_{n-1}) f(\Z_n) \right]  \, \middle| \, H_q(\Zbar_n; \B_{n-1}) g(\Z_n) \right\rangle_n.
        \end{split}
    \end{equation}
    The action on $n$-variate Laurent polynomials of the operator $\overline{D}_1$ is self adjoint with respect to the scalar product $\langle \, \cdot \, | \, \cdot \, \rangle_n$; see [Macdonald Section VI.9]. Thus, the right hand side is equal to
    \begin{equation}
        \begin{split}
            & \frac{1}{H_q(\A_{n-1},\B_{n-1})}  \left\langle  H_q(\Zbar_n;\A_{n-1}) f(\Z_n) \, \middle| \, \overline{D}_1 \left[  H_q(\Zbar_n; \B_{n-1}) g(\Z_n) \right] \right\rangle_n
            \\
            & = \frac{1}{H_q(A_{n-1}, \B_{n-1})}
            \\
            & \qquad \times \left\langle   f(\Z_n) H_q(\Z_n; \B_{n-1}) \, \middle| \, H_q( \Zbar_n; \B_{n-1})^{-1} \overline{D}_1 \left[  H_q(\Zbar_n; \B_{n-1}) g(\Z_n) \right] H_q(\Z_n;\A_{n-1}) \right\rangle_n
            \\
            & = \left\llangle   f(\Z_n)   \, \middle| \, H_q( \Zbar_n; \B_{n-1})^{-1} \overline{D}_1 \left[  H_q( \Zbar_n; \B_{n-1}) g(\Z_n) \right] \right\rrangle_n
            \\
            & = \left\llangle   f(\Z_n) \,  \middle| \,\overline{\mathfrak{D}}_{\B_{n-1}} \left[  g(\Z_n) \right] \right\rrangle_n.
        \end{split}
    \end{equation}
    Analogous calculations show that $\mathfrak{D}_{\A_{n-1}}, \mathfrak{D}_{\B_{n-1}}$ are adjoint with respect to the scalar product $\llangle \cdot \, | \, \cdot \rrangle_n$.
\end{proof}

\begin{lemma} \label{lem:prelim orthogonality}
    For all $\lambda,\mu \in \mathbb{Y}_n$ such that $\lambda_1 \neq \mu_1$ or $\lambda_n \neq \mu_n$, we have
    \begin{equation}
        \left\llangle \sqwF_\lambda( \Z_n| \A_{n-1}, \B_{n-1} ) \, \middle| \, \sqwF_\mu( \Z_n| \B_{n-1} , \A_{n-1} ) \right\rrangle_n = 0.
    \end{equation}
\end{lemma}
\begin{proof}
    This is an immediate consequence of \Cref{lem:adjoint} and of \Cref{thm:eigenrelation sqW lambda n}, \Cref{thm:eigenrelation sqW lambda 1}. For instance, considering the action of $\overline{\mathfrak{D}}_{\A_{n-1}}$, we have
    \begin{equation}
        \begin{split}
            &q^{-\lambda_1} \left\llangle \mathbb{F}_\lambda(\Z_n|\A_{n-1},\B_{n-1}) \, \middle| \, \mathbb{F}_\mu(\Z_n| \B_{n-1} , \A_{n-1} ) \right\rrangle_n\\
            &= \left\llangle \overline{\mathfrak{D}}_{\A_{n-1}} \mathbb{F}_\lambda(\Z_n|\A_{n-1},\B_{n-1}) \, \middle| \, \mathbb{F}_\mu(\Z_n| \B_{n-1} , \A_{n-1} ) \right\rrangle_n
            \\
            &= \left\llangle  \mathbb{F}_\lambda( \Z_n|\A_{n-1},\B_{n-1}) \, \middle| \, \overline{\mathfrak{D}}_{\B_{n-1}}\mathbb{F}_\mu(\Z_n| \B_{n-1} ,\A_{n-1} ) \right\rrangle_n
            \\
            & = q^{-\mu_1} \left\llangle \mathbb{F}_\lambda( \Z_n| \A_{n-1},\B_{n-1}) \, \middle| \, \mathbb{F}_\mu(\Z_n| \B_{n-1} , \A_{n-1} ) \right\rrangle_n,
        \end{split}
    \end{equation}
    which implies the claim when $\lambda_1 \neq \mu_1$. Similar basic arguments can be repeated to prove the claim in case $\lambda_n \neq \mu_n$.
\end{proof}

\begin{lemma} \label{lem:finite expansion inhom sqw}
    Fix integers $0<m<n$. Then, for every $\lambda \in \mathbb{Y}_m$, we have
    \begin{equation}
        \begin{split}
            &\sqwFBK_\lambda(\Y_m| \B_m, \A_m) 
            \\
            &= 
            \sum_{\mu \in \mathrm{Box}(\lambda_1,m) } \frac{\left\llangle \sqwF_\mu(\Z_n| \A_{n-1},\B_{n-1})   \, \middle| \, \sqwF_\lambda(\Z_n | \B_{n-1}, \A_{n-1} ) \right\rrangle_n }{\mathsf{c}_n(\mu|\A_{n-1},\B_{n-1})} \,  \sqwFBK_\mu(\Y_m | \B_m , \A_m ).
        \end{split}
    \end{equation}
\end{lemma}

\begin{proof}
    This is an immediate consequence of \Cref{lem:prelim orthogonality} and of the expansion \eqref{eq:expansion sqW infinite sum} for the function $\sqwFBK_\lambda(\Y_m|\B_m,\A_m)$. The summation of \eqref{eq:expansion sqW infinite sum} clearly runs over partitions $\mu$ such that $\ell(\mu)\le m <n$ and for such partitions we have
    \begin{equation}
        \prod_{k=1}^{n-1} \frac{(q;q)_\infty}{(a_k/b_k;q)_\infty}  \psi_\mu(\A,\B) = \frac{1}{\mathsf{c}_n(\mu|\A_{n-1},\B_{n-1})},
    \end{equation}
    completing the proof.
\end{proof}

\begin{lemma} \label{lem:expansion monomial inhom sqw}
    Let $\lambda \in \mathbb{Y}_n$. An expansion of the following form holds
    \begin{equation} \label{eq:expansion sqWBK}
        \sqwFBK_\lambda (\X_m| \A_m,\B_m ) = \sum_{\mu \subset \mathrm{Box}(\lambda_1,m)} \alpha_{\lambda,\mu} (\A_m,\B_m) \, m_\mu(\X_m), 
    \end{equation}
    where $m_\mu$ are the monomial symmetric polynomials. Coefficients $\alpha_{\mu,\lambda} (\A_m,\B_m)$ are analytic for $a_1,\dots,a_{m}$, $b_1,\dots,b_{m}$ in a neighborhood of the origin. Moreover for $a_1=\cdots=a_{m}=b_1=\cdots=b_{m}=0$ we have $\alpha_{\mu,\lambda} = 0$ unless $\mu \leq \lambda$ and $\alpha_{\lambda,\lambda}= 1$.
\end{lemma}
\begin{proof}
    Consider the branching factor $\sqwFBK_{\lambda/\mu}(x|\A_m,\B_m)$ of \eqref{eq:skew sqWBK}. It is clear that
    \begin{equation}
        \deg_x\left( \sqwFBK_{\lambda/\mu}(x|\A_m,\B_m) \right) = |\lambda| - | \mu| + \sum_{r\ge 1} \mu_{r}-\lambda_{r+1} = \lambda_1
    \end{equation}
    and that the coefficients of $\sqwFBK_{\lambda/\mu}(x|\A_m,\B_m)$ are analytic functions of parameters $\A_m,\B_m$ as long as $|a_1|,\dots,|a_m|, |b_1|,\dots,|b_m|<1$. Then, the branching rule \eqref{eq:branching rule} implies that
    \begin{equation}
        \deg_{x_i} \left( \sqwFBK_{\lambda}( \X_m |\A_m,\B_m) \right) = \lambda_1,
        \qquad \text{for all } i=1,\dots,m,
    \end{equation}
    by symmetry and this proves the expansion \eqref{eq:expansion sqWBK}. Setting all parameters $\A_m,\B_m$ to zero, the spin $q$-Whittaker polynomial $\sqwFBK$ reduces to the $q$-Whittaker polynomial $P$, as in \Cref{eq:limit sqW to qW}, and we have (\cite[VI.4]{Macdonald1995})
    \begin{equation}
        P_\lambda(\X_m;q,0) = m_\lambda(\X_m) + \sum_{\mu < \lambda} c_{\lambda,\mu} \, m_\mu(\X_n),
    \end{equation}
    where $c_{\lambda,\mu}$ are polynomials in $q$. This completes the proof. 
\end{proof}

We are now in the position to prove the orthogonality of spin $q$-Whittaker polynomials.

\begin{proof}[Proof of \Cref{thm:orthogonality inhom sqW}]
    Fix $L \ge 0$ and $\lambda$ such that $\lambda_1 \le L$. By \Cref{lem:finite expansion inhom sqw}, we have the expansion 
    \begin{equation} \label{eq:expansion sqWFBK beta}
        \sqwFBK_\lambda(\Y_m|\B_m, \A_m) = \sum_{\nu \subset \mathrm{Box}(L,m)} \beta_{\lambda,\nu} (\A_{n-1},\B_{n-1}) \, \mathbb{F}_\nu(\Y_m | \B_m, \A_m ),
    \end{equation}
    where
    \begin{equation} \label{eq:beta}
            \beta_{\lambda,\nu} (\A_{n-1},\B_{n-1}) = \frac{1}{\mathsf{c}_n(\nu|\A_{n-1},\B_{n-1})}  \left\llangle \sqwF_\nu(\Z_n| \A_{n-1},\B_{n-1})   \, \middle| \, \sqwF_\lambda(\Z_n | \B_{n-1}, \A_{n-1} ) \right\rrangle_n.
    \end{equation}
    Coefficients $\beta_{\lambda,\nu} (\A_{n-1},\B_{n-1})$
    are analytic functions of parameters $\A_{n-1}, \B_{n-1}$ for $|a_1|,\dots,|a_{n-1}|$, $|b_1|,\dots , |b_{n-1}|<1$ which can be seen by a direct inspection of each term in the right hand side of \eqref{eq:beta}. 
    Using \Cref{lem:expansion monomial inhom sqw} in the expansion \eqref{eq:expansion sqWFBK beta}, we reduce the relation to
    \begin{equation} \label{eq:relation alpha beta}
        \begin{split}
            &\sum_{\mu \in \mathrm{Box}(L,m)} \alpha_{\lambda,\mu} (\A_{m},\B_{m}) m_\mu(\Y_m) 
            \\
            & \qquad = \sum_{\nu \in \mathrm{Box}(L,m)} \sum_{\mu \in \mathrm{Box}(L,m)} \, \beta_{\lambda,\nu} (\A_{n-1},\B_{n-1}) \alpha_{\nu,\mu} (\A_m,\B_m) \, m_\mu(\Y_m).
        \end{split}
    \end{equation}
    Define the finite dimensional square matrices 
    \begin{equation}
        \begin{aligned}
            \alpha (\A_{n-1},\B_{n-1}) = \left[ \alpha_{\lambda,\mu} (\A_{n-1},\B_{n-1}) \right]_{\lambda,\mu \in \mathrm{Box}(L,n)},
            \\
            \beta(\A_{n-1},\B_{n-1}) = \left[ \beta_{\lambda,\mu} (\A_{n-1},\B_{n-1}) \right]_{\lambda,\mu \in \mathrm{Box}(L,n)}.
        \end{aligned}
    \end{equation}
    Then, relation \eqref{eq:relation alpha beta} implies the matrix relation
    \begin{equation} \label{eq:matrix alpha beta relation}
        \alpha (\A_{n-1},\B_{n-1}) = \beta (\A_{n-1},\B_{n-1}) \alpha (\A_{n-1},\B_{n-1}).
    \end{equation}
    For $a_1=\cdots=a_{n-1}=b_1=\cdots=b_{n-1}=0$, the matrix $\alpha$ is upper triangular and $\alpha_{\lambda,\lambda}=1$ and as a result it is invertible. Since each coefficient of $\alpha(\A_{n-1},\B_{n-1})$ is analytic for $|a_i|,|b_j|<1$, there exists $\varepsilon>0$ such that
    \begin{equation}
        \det\left[ \alpha (\A_{n-1},\B_{n-1}) \right] > 0, \qquad \forall |a_1|,\dots,|a_{n-1}|,|b_1|,\dots,|b_{n-1}|\le \varepsilon.
    \end{equation}
    Then, $\alpha (\A_{n-1},\B_{n-1})$ remains invertible in the same region and relation \eqref{eq:matrix alpha beta relation} implies that
    \begin{equation}
        \beta (\A_{n-1},\B_{n-1}) = \mathrm{Id}, \qquad \forall |a_1|,\dots,|a_{n-1}|,|b_1|,\dots,|b_{n-1}|\le \varepsilon.
    \end{equation}
    Since $\beta (\A_{n-1},\B_{n-1})$ is analytic we have $\beta(\A_{n-1},\B_{n-1}) = \mathrm{Id}$ for all $|a_1|,\dots,|a_{n-1}|,|b_1|,\dots |b_{n-1}| < 1$, showing that in the same region of parameters the desired orthogonality relation \eqref{eq:orthogonality sqW} holds for all partitions $\lambda , \mu \in \mathbb{Y}_m$. Setting $m=n-1$, we see that to conclude the proof we are only left to check that the orthogonality relation \eqref{eq:orthogonality sqW} also holds when at least one of the two partitions $\lambda,\mu$ has length equal to $n$. Since the case when $\lambda_n \neq \mu_n$ is covered by \Cref{lem:prelim orthogonality}, we assume that $\lambda_n = \mu_n$. In this case, using the shift property \eqref{eq:shift property}, we have
    \begin{equation} \label{eq:scalar prod with shift property}
        \begin{split}
            &\left\llangle \sqwF_\lambda(\Z_n|\A_{n-1},\B_{n-1}) \middle| \sqwF_\mu (\Z_n| \B_{n-1} , \A_{n-1} ) \right\rrangle_n
            \\
            & = \left\llangle \left( z_1 \cdots z_n \right)^{\lambda_n} \sqwF_{\tilde{\lambda}}(\Z_n|\A_{n-1},\B_{n-1}) \middle| \left(z_1 \cdots z_n \right)^{\mu_n} \sqwF_{\tilde{\mu}} (\Z_n| \B_{n-1} , \A_{n-1} ) \right\rrangle_n
            \\
            & = \left\llangle  \sqwF_{\tilde{\lambda}}(\Z_n|\A_{n-1},\B_{n-1}) \middle| \sqwF_{\tilde{\mu}} (\Z_n| \B_{n-1} , \A_{n-1} ) \right\rrangle_n
            \\
            & = \mathbf{1}_{\tilde{\lambda}=\tilde{\mu}} \mathsf{c}_n( \tilde{\lambda} | \A_{n-1}, \B_{n-1} ) ,
        \end{split}
    \end{equation}
    where we denoted 
    \begin{equation}
        \tilde{\lambda}=(\lambda_1 - \lambda_n, \dots, \lambda_{n-1} -\lambda_n, 0),
        \qquad
        \tilde{\mu}=(\mu_1 - \mu_n, \dots, \mu_{n-1} -\mu_n, 0).
    \end{equation} 
    This concludes the proof since 
    \begin{equation}
        \mathsf{c}_n( \tilde{\lambda} | \A_{n-1}, \B_{n-1} ) = \mathsf{c}_n( \lambda | \A_{n-1}, \B_{n-1} ),
    \end{equation}
    as it can be verified from the expression \eqref{eq:c AB}.
\end{proof}

\begin{remark}
    Setting $a_i=b_i=\sqrt{q}$, assuming $q\in (0,1)$, makes the expression \eqref{eq:orthogonality sqW} symmetric and the polynomials $\sqwF_\lambda(\Z_n|(\sqrt{q},\dots,\sqrt{q}),(\sqrt{q},\dots,\sqrt{q}))$ become self-orthonormal. It is not clear if this simplification in the $L^2$ norms $\mathsf{c}_n(\lambda|(\sqrt{q},\dots,\sqrt{q}),(\sqrt{q},\dots,\sqrt{q}))$ has any deeper implication.
\end{remark}

\subsection{Proofs of \Cref{thm:sqW are a basis} and \Cref{thm:sqWBK are a basis}} \label{subs: proof sqW are a basis}

Let us now prove that the spin $q$-Whittaker polynomials are a basis of the space of $n$-variate symmetric polynomials, denoted by $\mathsf{Sym}_n$. First, we report the next lemma.

\begin{lemma} \label{lem:sqw generate}
    Fix $n\in \mathbb{N}$. The families of symmetric polynomials $\{ \sqwF_\lambda(\X_n|\A_{n-1},\B_{n-1}) \}_{\lambda \in \mathbb{Y}_n}$ and $\{ \sqwFBK_\lambda(\X_n|\A_{n},\B_{n}) \}_{\lambda \in \mathbb{Y}_n}$ generate the space $\mathsf{Sym}_n$.
\end{lemma}

\begin{proof}
    Let us first prove that the family of spin $q$-Whittaker polynomials $\{ \sqwFBK_\lambda(\X_n|\A_{n},\B_{n}) \}_{\lambda \in \mathbb{Y}_n}$ generate the space of $n$-variate of symmetric polynomials. From the Pieri rule \eqref{eq:pieri rule}, we have
    \begin{equation}
        \left( \sum_{k=0}^n e_k(\X_n)(-u)^k \right)
        \sqwFBK_\mu(\X_n|\A_n,\B_n)
        =\sum_{\lambda} c_{\lambda,\mu} (u)  \sqwFBK_\lambda(\X_n|\A_n,\B_n),
    \end{equation}
    where $e_k(\X_n)$ is the $k$-th elementary symmetric polynomials in $n$ variables and 
    \begin{equation}
        c_{\lambda,\mu} (u) = E(\A_n;u) \shlF_{\lambda'/\mu'} (u|\A,\B).
    \end{equation}
    This implies that, for any $\nu \in \mathbb{Y}$, we can express the product of elementary symmetric polynomials
    \begin{equation}
        e_\nu(\X_n) = e_{\nu_1}(\X_n) \cdots e_{\ell(\nu)}(\X_n),
    \end{equation}
    as a linear combination of spin $q$-Whittaker polynomials. Since the set of symmetric polynomials $\{ e_\nu(\X_n) \}_{\nu \in \mathbb{Y}}$ generates $\mathsf{Sym}_n$, so do the $n$-variate spin $q$-Whittaker polynomials $\{ \sqwFBK_\lambda(\X_n|\A_{n},\B_{n}) \}_{\lambda \in \mathbb{Y}_n}$.

    The analogous statement for the family $\{ \sqwF_\lambda(\X_n|\A_{n-1},\B_{n-1}) \}_{\lambda \in \mathbb{Y}_n}$ can be proved in the same way using the corresponding Pieri rules.
\end{proof}

\begin{proof}[Proof of \Cref{thm:sqW are a basis}]
    This is a straightforward consequence of \Cref{thm:orthogonality inhom sqW} and of \Cref{lem:sqw generate}. An alternative proof is as follows\footnote{Suggested by Travis Scrimshaw.}. The spin $q$-Whittaker polynomials are linearly independent as a result of \Cref{thm:orthogonality inhom sqW}. Moreover, by \Cref{lem:expansion monomial inhom sqw} and the shift property of \Cref{prop:shift property} we have
    \begin{equation}
        \sqwF_\lambda(\X_n|\A_{n-1} , \B_{n-1}) \in \mathrm{Span} \{ m_\nu(\X_n) \}_{\nu \in \mathrm{Box}(L,n)},
        \qquad \text{for all } \lambda \in \mathrm{Box}(L,n).
    \end{equation}
    This implies that, for all $L>0$, we have
    \begin{equation}
        \mathrm{Span} \{ \sqwF_\lambda(\X_n|\A_{n-1} , \B_{n-1}) \}_{\lambda \in \mathrm{Box}(L,n)} = \mathrm{Span} \{ m_\nu(\X_n) \}_{\nu \in \mathrm{Box}(L,n)},
    \end{equation}
    which proves that $\{\sqwF_\lambda(\X_n|\A_{n-1}, \B_{n-1})\}_{\lambda \in \mathbb{Y}_n}$ generates $\mathsf{Sym}_n$.
\end{proof}

\begin{proof}[Proof of \Cref{thm:sqWBK are a basis}]
    We proceed by contradiction. For this assume that there exists a set $S \subset \mathbb{Y}_n$ and non-zero coefficients $\{\gamma_\lambda\}_{\lambda \in S}$ such that
    \begin{equation} \label{eq:linear combination 0}
        \sum_{\lambda \in S} \gamma_\lambda \sqwFBK_\lambda(\X_n |\A_n, \B_n) = 0.
    \end{equation}
    Replacing $\sqwFBK_\lambda(\X_n |\A_n, \B_n)$ with $\sqwF_\lambda(\X_{n+1} |\A_n, \B_n)$ in the left hand side of \eqref{eq:linear combination 0} we obtain the symmetric polynomial in $n+1$ variables
    \begin{equation} \label{eq:expansion f 1}
        f(\X_{n+1})=\sum_{\lambda \in S} \gamma_\lambda \sqwF_\lambda(\X_{n+1} |\A_n, \B_n).
    \end{equation}
    By \Cref{prop:relation sqW and sqWBK}, the polynomial $f(\X_{n+1})$ reduces to the left hand side of $\eqref{eq:linear combination 0}$ setting $x_i=0$ for any $i=1,\dots,n+1$. As a result the monomial $x_1 \cdots \, x_{n+1}$ divides the polynomial $f(\X_{n+1})$ and we can write
    \begin{equation}
        f(\X_{n+1}) = x_1 \cdots \, x_{n+1} \, \tilde{f}(\X_{n+1}),
    \end{equation}
    where $\tilde{f}$ is another symmetric polynomial in $n+1$ variables. Expanding $\tilde{f}$ in the basis of $\sqwF$ polynomials, we get
    \begin{equation}
        \tilde{f}(\X_{n+1}) = \sum_{\mu} \tilde{\gamma}_\mu \, \sqwF_\mu(\X_{n+1}|\A_n, \B_n),
    \end{equation}
    for some coefficients $\tilde{\gamma}_\mu$. This yields the expansion
    \begin{equation} \label{eq:expansion f 2}
        f(\X_{n+1}) = \sum_{\mu} \tilde{\gamma}_\mu \, \sqwF_{(\mu_1+1 , \dots, \mu_{n+1}+1)}(\X_{n+1}|\A_n, \B_n),
    \end{equation}
    using the shift property \eqref{eq:shift property}. Since polynomials $\sqwF$ form a basis, the two expansions \eqref{eq:expansion f 1} and \eqref{eq:expansion f 2} need to be equal, which is a contradiction since the set $S$ contains only partitions of length strictly less than $n+1$.
\end{proof}

\section{Special cases} \label{sec:special cases}

In this section we consider several special cases and scaling limits of the spin $q$-Whittaker polynomials and the corresponding specializations of \Cref{thm:orthogonality inhom sqW} and \Cref{thm:sqW are a basis}. To avoid exceedingly technical arguments, in \Cref{subs:spin Whittaker}, which is dedicated to the spin Whittaker functions, the discussion will be kept at a formal level.

\subsection{Inhomogeneous symmetric Grothendieck polynomials}

The first specialization of the spin $q$-Whittaker polynomials we consider is setting $q=0$ and $a_1=\cdots=a_{n-1}=0$, leading to the definitions of polynomials
\begin{equation}
    G_\lambda(\X_n|\B_{n-1}) = \sqwF_\lambda(\X_n| \A_{n-1}, \B_{n-1})\big|_{q=0, \,a_1=\cdots=a_{n-1}=0}
\end{equation}
and their variant 
\begin{equation}
    \overline{G}_\lambda(\X_n|\B_{n-1}) = \sqwF_\lambda(\X_n| \B_{n-1}, \A_{n-1})\big|_{q=0, \, a_1=\cdots=a_{n-1}=0}.
\end{equation}
The polynomials $G$ are the (inhomogeneous) symmetric Grothendieck polynomials \cite{Chan_Pflueger_Grothendieck,Hwang_et_al_Grothendieck_combinatorial}; see \cite{Fomin_Kirillov_Grothendieck} for the case $b_1=\cdots=b_n=\beta$ and also \cite{Lascoux_Schutzemberger_Grothendieck,Buch_Grothendieck} for the case $b_1=\cdots=b_n=1$. They are more commonly defined through the bialternant formulas
\begin{equation} \label{eq:G bialternant}
    G_\lambda(\X_n|\B_{n-1}) = \frac{\det\left[ x_i^{\lambda_j +n-j} \prod_{k=1}^{j-1}(1- x_i b_k) \right]_{i,j=1,\dots,n}}{\det \left[ x_i^{n-j} \right]_{i,j=1,\dots,n}}.
\end{equation}
Notice that, as a consequence of \Cref{prop:reverse symmetry}, from \eqref{eq:G bialternant} one gets
\begin{equation} \label{eq:G star bialternant}
    \overline{G}_\lambda(\X_n|\B_{n-1}) = \frac{\det\left[ x_i^{\lambda_j} \prod_{k=1}^{n-j}(x_i-b_{n-k}) \right]_{i,j=1,\dots,n}}{\det \left[ x_i^{n-j} \right]_{i,j=1,\dots,n}}.
\end{equation}
To verify that the polynomials $G$ defined specializing the spin $q$-Whittaker polynomials correspond to the Grothendieck polynomials, one can check that the branching rule \eqref{eq:branching rule} reduces, after setting $q=a_1=\cdots=a_n=0$, to the set-valued tableaux combinatorial formula given in \cite[Theorem 3.2]{Hwang_et_al_Grothendieck_combinatorial} or \cite[Theorem 3.4]{Chan_Pflueger_Grothendieck}, which generalizes previously known combinatorial formulas for the case $b_1=\cdots=b_n=1$ \cite{Buch_Grothendieck}; see also \cite[Theorem 4.2]{yeliussizov2015duality}. Alternatively one can observe that, again as a consequence of the branching rule \eqref{eq:branching rule}, we have
\begin{equation} \label{eq:monomial expansion Grothendieck}
    G_\lambda(\X_n|\B_{n-1}) = m_\lambda(\X_n) + \sum_{\mu>\lambda} c_{\lambda,\mu} m_\mu(\X_n),
    \qquad 
    \overline{G}_\lambda (\X_n|\B_{n-1}) = m_\lambda(\X_n) + \sum_{\mu<\lambda} \overline{c}_{\lambda,\mu} m_\mu(\X_n),
\end{equation}
for some coefficients $\{c_{\lambda,\mu}\},\{\overline{c}_{\lambda,\mu}\}$ and specializing the Cauchy identities \eqref{eq:inhom CI sqWBK sqWBK}, we have
\begin{equation} \label{eq:CI Grothendieck}
    \sum_{\lambda} G_\lambda(\X_n|\B_{n-1}) \overline{G}_\lambda (\Y_n|\B_{n-1}) = \frac{\prod_{i=1}^n \prod_{j=1}^{n-1}(1-b_j x_i)}{\prod_{i=1}^n \prod_{j=1}^n(1-x_iy_j)}.
\end{equation}
Properties \eqref{eq:monomial expansion Grothendieck}, \eqref{eq:CI Grothendieck}, which are also satisfied by Grothendieck polynomials (see e.g. \cite[eq. (3.7)]{Gavrilova_Petrov_Tilted}), uniquely identify the polynomials $G,\overline{G}$ and therefore \eqref{eq:G bialternant}, \eqref{eq:G star bialternant} hold.

\begin{remark}
    In \cite{Motegi_Sakai_2013} they also define the functions $\overline{G}_\lambda(\Y_n|\B_{n-1})$, which differ by ours, in the case $b_1=\cdots=b_{n-1}=-\beta$ by a multiplication of a factor $\prod_{i=1}^n (1-\beta/y_i)^{n-1}$.
\end{remark}

\medskip

To specialize the orthogonality result of \Cref{thm:orthogonality inhom sqW} to the inhomogeneous symmetric Grothendieck polynomials, we define the scalar product 
\begin{equation}
    \left\langle f|g \right\rangle_{n,\B_{n-1}} = \frac{1}{n!} \int_{\mathbb{T}^n} f(\Z_n) g(\Zbar_n)    \frac{\prod_{1 \le i \neq j \le n} (1-z_i/z_j)}{ \prod_{i=1}^n \prod_{j=1}^{n-1} (1-z_i b_j) } \prod_{i=1}^n \frac{ \diff z_i }{2 \pi \mathrm{i} z_i}.
\end{equation}

\begin{theorem} \label{thm:orthogonality Grothendieck polys}
    We have
    \begin{equation} \label{eq:orthogonality Grothendieck polys}
        \left\langle G_\lambda(\Z_n|\B_{n-1}) |\overline{G}_\lambda(\Z_n|\B_{n-1}) \right\rangle_{n,\B_{n-1}} = \mathbf{1}_{\lambda=\mu}.
    \end{equation}
\end{theorem}

We have not been able to locate the result of \Cref{thm:orthogonality Grothendieck polys} in literature. In the Schur polynomials case, when $b_1=\cdots = b_{n-1}=0$, the orthogonality \eqref{eq:orthogonality Grothendieck polys} corresponds to the orthogonality of characters of irreducible polynomial representations of the unitary group $U(n)$. It would be interesting to find if a similar representation theoretic interpretation can be given to the orthogonality result of \Cref{thm:orthogonality Grothendieck polys}.

\begin{remark}
    The closest result to \Cref{thm:orthogonality Grothendieck polys} we were able to find in literature is in \cite{Motegi_Sakai_2013}, where the polynomials $G_\lambda(\X_n|\B_{n-1})$ were defined for the case $b_1=\cdots=b_{n-1}=-\beta$. There, in \cite[Theorem 5.6]{Motegi_Sakai_2013}, authors  give the orthogonality result for the pair $G_\lambda, \overline{G}_\mu$ 
    \begin{equation}
        \sum_{\{ \Z_n \} } w_M(\Z_n) G_\lambda(\Z_n|(-\beta,\dots,-\beta)) \overline{G}_\lambda(\Zbar_n|(-\beta,\dots,-\beta))  = \mathbf{1}_{\lambda=\mu},
    \end{equation}
    where the sum runs over the set of roots of a certain family of polynomial equations (the Bethe ansatz equations) and partitions $\lambda,\mu$ are bounded by quantities depending on $n,M$. Above the weight $w_M$ depends on the positive integer $M$ and letting $M\to +\infty$ it is natural to conjecture that one should recover \eqref{eq:orthogonality Grothendieck polys}, for the particular case $b_1=\cdots=b_{n-1}=-\beta$.
\end{remark}

\subsection{Interpolation $q$-Whittaker polynomials}

Following \cite{Korotkikh2024}, we define the \emph{interpolation $q$-Whittaker polynomials}
\begin{equation}
    P_\lambda(\X_n|\B_{n-1}) = \sqwF_\lambda(\X_n| \A_{n-1}, \B_{n-1})\big|_{a_1=\cdots=a_{n-1}=0}
\end{equation}
and their dual variant
\begin{equation}
    \overline{P}_\lambda(\X_n|\B_{n-1}) = \sqwF_\lambda(\X_n| \B_{n-1} , \A_{n-1} )\big|_{a_1=\cdots=a_{n-1}=0}.
\end{equation}
These are at the same time $q$-deformation of the Grothendieck polynomials discussed in the previous subsection and multi parameter inhomogenous generalizations of the $q$-Whittaker polynomials. In \cite{Korotkikh2024}, it was shown that the polynomials $P_\lambda(\X_n|\B_{n})$ (which one defines setting $a_1=\cdots=a_n=0$ in the polynomials $\sqwFBK(\X_n|\A_n,\B_n)$) are unique solutions of certain vanishing conditions and as a result they are a linear basis of $\mathsf{Sym}_n$.

\begin{example}
    For the case $n=2$ and $\lambda=(2,0)$ we have
    \begin{equation}
        \begin{split}
            P_{(2,0)}(x_1,x_2|b_1) &= \frac{\left(1-q^2\right) x_2 x_1 \left(1-b_1 x_1\right)}{1-q}+x_2^2 \left(1-b_1 x_1\right) \left(1-b_1 q x_1\right)+x_1^2
            \\
            & = P_{(2,0)}(x_1,x_2;q,0) - b_1 (1+q) P_{(2,0)}(x_1,x_2;q,0) + b_1^2 q  P_{(2,2)}(x_1,x_2;q,0),
        \end{split}
    \end{equation}
    where again $P_{\lambda}(\X_n;q,0)$ denote the $q$-Whittaker polynomials.
\end{example}

To specialize \Cref{thm:orthogonality inhom sqW} to the interpolation $q$-Whittaker polynomials, we define the scalar product 
\begin{equation}
    \left\llangle f|g \right\rrangle_{n,\B_{n-1}} = \frac{1}{n!} \int_{\mathbb{T}^n} f(\Z_n) g(\Zbar_n)    \frac{\prod_{1 \le i \neq j \le n} (z_i/z_j;q)_\infty}{ \prod_{i=1}^n \prod_{j=1}^{n-1} (z_i b_j;q)_\infty } \prod_{i=1}^n \frac{ \diff z_i }{2 \pi \mathrm{i} z_i}.
\end{equation}
\begin{theorem}
    The families of symmetric polynomials $\{P_\lambda(\X_n|\B_{n-1})\}_{\lambda \in \mathbb{Y}_n}$ and $\{\overline{P}_\lambda(\X_n|\B_{n-1})\}_{\lambda \in \mathbb{Y}_n}$ are linear basis of $\mathsf{Sym}_n$. Moreover, we have 
    \begin{equation}
        \left\llangle P_\lambda(\Z_n|\B_{n-1}) \middle| \overline{P}_\lambda(\Z_n|\B_{n-1}) \right\rrangle_{n,\B_{n-1}}        = \mathbf{1}_{\lambda=\mu} \prod_{k=1}^{n-1} \frac{(q;q)_{\lambda_k-\lambda_{k+1}}}{(q;q)_\infty}.
    \end{equation}
\end{theorem}

\subsection{Formal reduction to inhomogeneous spin Whittaker functions} \label{subs:spin Whittaker}
In this section we define the \emph{inhomogeneous spin Whittaker} functions, which are inhomogeneous variants of the functions defined in \cite{MucciconiPetrov2020}. They are multi parameter generalizations of the $GL(n)$-Whittaker functions \cite{Kostant_Whittaker,givental1996stationary} in the same way the spin $q$-Whittaker polynomials generalize the $q$-Whittaker polynomials. As briefly discussed below, the scaling limit linking the spin $q$-Whittaker polynomials to the spin Whittaker functions, unlike in the previous subsections, is singular as it involves in multiple places the convergence of the $q$-Pochhammer symbols to the Euler Gamma function. For this reason, to avoid potentially lengthy proofs, we will keep the discussion formal and one can treat the statements presented as conjectures.

\medskip

We define the Weyl chamber of $\mathbb{R}_{\ge 1}^n$ as
\begin{equation}
    \mathcal{W}_n = \left\{ \L_n \in \mathbb{R}_{\ge 1}^n: L_n \le \cdots \le L_1 \right\}.
\end{equation}
Between elements of the Weyl chambers $\L_{n-1}'\in \mathcal{W}_{n-1}$ and $\L_n \in \mathcal{W}_{n}$, we define the interlacing relations
\begin{equation}
    \L'_{n-1} \preceq \L_n, \qquad \text{if} \qquad L_{i+1} \le L'_i \le L_i, \qquad \text{for all } i=1,\dots,n-1.
\end{equation}
Define the inhomogeneous skew spin Whittaker function
\begin{equation}
    \mathfrak{f}_X( \L_{n-1}' ; \L_n  |\AA_{n-1},\BB_{n-1}) =  \mathbf{1}_{\L_{n-1}' \preceq \L_n} \left( \frac{L_1\cdots L_n}{L_1'\cdots L_{n-1}' } \right)^{-X} \prod_{i=1}^{n-1} \mathcal{A}_{A_i,B_i;X}(L_i,L_i',L_{i+1}),
\end{equation}
where 
\begin{equation}
    \mathcal{A}_{A,B;X}(u,v,z) = \frac{1}{\mathrm{Beta}(A-X,B+X)} \left( 1- \frac{v}{z} \right)^{A-X-1} \left( 1- \frac{u}{v} \right)^{B+X-1} \left( 1- \frac{u}{z} \right)^{1-A-B},
\end{equation}
for all $1\le u<v<z$ and $A>X>-B$. Above Beta denotes Euler's Beta function. Through a branching rule we define the inhomogeneous spin Whittaker function as
\begin{equation}
    \mathfrak{f}_{\XX_n}( \L_n  |\AA_{n-1},\BB_{n-1}) = \int_{\L_{n-1}' \preceq \L_n} \mathfrak{f}_{\XX_{n-1}} (\L_{n-1}'|\tau \AA_{n-1},\BB_{n-2}) \mathfrak{f}_{X_n}( \L_{n-1}' ; \L_n  |\AA_{n-1},\BB_{n-1}) \prod_{i=1}^n \frac{\diff L_i}{L_i},
\end{equation}
where variables $\XX_n$ and parameters $\AA_{n-1}, \BB_{n-1}$ are complex numbers satisfying the conditions
\begin{equation}
    \Re\{ A_i - X_j \}, \Re\{ B_i + X_j \} > 0, \qquad \text{for all } i=1,\dots,n-1 \text{ and } j=1,\dots,n.
\end{equation}

The spin Whittaker functions should be recovered as a limit of the spin $q$-Whittaker polynomials as
\begin{equation} \label{eq:limit sqW to SW}
    \lim_{q\to 1}  \left( \log q^{-1} \right)^{-\binom{n}{2} } \sqwF_\lambda(\X_n|\A_{n-1}, \B_{n-1}) = \mathfrak{f}_{\XX_n}(\L_n|\AA_{n-1},\BB_{n-1}),
\end{equation}
under the scaling
\begin{equation} \label{eq:scaling sqW to SW}
    x_i=q^{X_i}, \qquad a_i = q^{A_i}, \qquad b_i = q^{B_i}, \qquad \lambda_i = \lfloor \log_q (1/L_i) \rfloor.
\end{equation}
In \cite{MucciconiPetrov2020} the functions $\mathfrak{f}_{\XX_n}( \L_n  |\AA_{n-1},\BB_{n-1})$ were defined in the particular (homogeneous) case $A_1=\dots=A_{n-1}=B_1=\cdots=B_{n-1}=S>0$ and the limit \eqref{eq:limit sqW to SW} was shown to hold locally uniformly for $\L_n \in \mathcal{W}_n$. These functions were shown to satisfy integral identities which are continuous analogs of the Cauchy identities \eqref{eq:inhom CI sqW sqW}. Furthermore it was shown that they are eigenfunction of the second order differential operator
\begin{equation}
    -\frac{1}{2} \sum_{i=1}^n \partial_{u_i}^2 + \sum_{1 \le i < j \le n} S^{-2(j-i)} e^{u_j-u_i} (S-\partial_{u_i}) (S+\partial_{u_j}),
\end{equation}
where $L_j=S^{N+1-2j}e^{u_j}$ and with eigenvalue
\begin{equation}
    -\frac{1}{2}\left( X_1^2 + \cdots + X_n^2 \right).
\end{equation} 
Such eigenrelations arises as a scaling limit of the Pieri rules \eqref{eq:pieri rule}, following the recipe of \cite{GerasimovLebedevOblezin2011}, and as a result their inhomogeneous extensiond should exist, although we do not compute them here.

\medskip

We conclude addressing the orthogonality of the functions $\mathfrak{f}_{\XX_n}$, which was conjectured in a special case in \cite{MucciconiPetrov2020}. Define the Sklyanin-type density 
\begin{equation}
    \mathfrak{M}(\ZZ_n|\AA_{n-1},\BB_{n-1}) = \frac{\prod_{i=1}^n \prod_{j=1}^{n-1} \Gamma(A_j - Z_i) \Gamma(B_j+Z_i) }{ \prod_{i=1}^{n-1} \Gamma(A_i+B_i) \prod_{i=1}^{n-1}\prod_{j=1}^{n-1} \Gamma(A_i +B_j) } \mathfrak{M}(\ZZ_n),
\end{equation}
where
\begin{equation}
    \mathfrak{M}(\ZZ_n) = \frac{1}{n! (2 \pi \mathrm{i})^n} \prod_{1\le i \neq j \le n} \frac{1}{\Gamma(Z_i -Z_j )}.
\end{equation}
A formal limit of \Cref{thm:orthogonality inhom sqW} should show that, for all $\L_n, \L_n' \in \WeylCh_n$, we have
\begin{equation} \label{eq:orthogonality SW}
    \begin{split}
        &\int_{(\mathrm{i}\mathbb{R})^n} \mathfrak{f}_{\ZZ_n}(\L_n|\AA_{n-1},\BB_{n-1}) \mathfrak{f}_{-\ZZ_n}(\L_n'|\BB_{n-1}, \AA_{n-1}) \mathfrak{M}(\ZZ_n|\AA_{n-1},\BB_{n-1}) \diff Z_1 \cdots \diff Z_n 
        \\
        & \qquad \qquad \qquad \qquad \qquad \qquad \qquad \qquad \qquad \qquad \qquad = \prod_{i=1}^{n-1} \left( 1 - \frac{L_{i+1}}{L_{i}} \right)^{1-A_i-B_i} \delta_{\L_n - \L_n'}.
    \end{split}
\end{equation}
A rigorous proof of \eqref{eq:orthogonality SW} addressing the (weak) convergence of left and right hand side of \eqref{eq:orthogonality sqW} under the scaling \eqref{eq:scaling sqW to SW} would require the use of convergence and decay estimates of $q$-Pochhammer symbols as those of \cite[Appendix C]{MucciconiPetrov2020}, \cite[Appendix B]{IMS_KPZ_free_fermions} and \cite[Section 4]{BorodinCorwin2011Macdonald}.

\section{A concluding remark} \label{sec:conclusion}

The orthogonality relation we prove in \Cref{thm:orthogonality inhom sqW} was initially conjectured in \cite{MucciconiPetrov2020} from integrable probabilistic argument. The spin $q$-Whittaker polynomials $\sqwF$, as proven in \cite{MucciconiPetrov2020}, can be used to describe the joint probability distribution of a stochastic particle system called $q$-Hahn TASEP introduced in \cite{Povolotsky2013,Corwin2014qmunu}. The probability distribution of a tagged particle in the $q$-Hahn TASEP, under the assumption of step initial conditions, was written in \cite{imamura2019stationary} as an $n$-fold contour integral. In the notation of this paper, this contour integral is\footnote{To match this expression with the one of \cite[Section 10.1]{MucciconiPetrov2020}, one should set $x_i=s,a_i=b_i=-s,y_i=y$.}
\begin{equation} \label{eq:probability tagged particle}
    \begin{split}
        &\mathbb{P}\left( \text{$n$-th particle at time $t$ is at position $\ell$}  \right)
        \\
        & = \prod_{i=1}^{n-1} \frac{(q;q)_\infty}{(a_i b_i;q)_\infty} 
        \left\llangle \frac{ H_q( \Z_n ; \Y_t ) H_q( \X_n ; \B_t ) }{ H_q( \Z_n ; \B_t ) H_q( \X_n ; \Y_t ) } \middle| \left( XZ \right)^\ell (XZ;q)_\infty \frac{ H_q(\X_n;\Z_n) H_q(\A_{n-1};\B_{n-1}) }{H_q(\X_n;\B_{n-1}) H_q(\A_{n-1};\Z_{n})} \right\rrangle_n,
    \end{split}
\end{equation}
where $XZ=x_1 z_1 \cdots x_n z_n$.
Assuming the orthogonality \eqref{eq:orthogonality sqW} and the shift property \eqref{eq:shift property}, the above formula could also (and more naturally) be derived from the results of \cite{MucciconiPetrov2020}, since left and right hand side of \eqref{eq:probability tagged particle} are proven to be equal to
\begin{equation}
    \frac{ H_q(\X_n;\B_t) H_q(\A_{n-1} ; \Y_t) }{ H_q(\X_n;\Y_t)  H_q(\A_{n-1} ; \B_t) } \sum_{ \substack{ \lambda \in \mathbb{Y}_n \\ \lambda_n=\ell } } \sqwF_\lambda(\X_n| \A_{n-1}, \B_{n-1}) \sqwFBK_\lambda^*(\Y_t| \B_t ,\A_t).
\end{equation}
This derivation is explained in \cite[Section 10.1]{MucciconiPetrov2020} and, in light of \Cref{thm:orthogonality inhom sqW}, is now rigorous. The $n$-fold integral formula \eqref{eq:probability tagged particle} can be turned into an $n$-fold integral representation for the $q$-Laplace trasform of the probability mass function as 
\begin{equation} \label{eq:laplace transform}
    \begin{split}
        &\mathbb{E} \left[ \frac{1}{( \zeta q^{ \text{position of the $n$-th particle at time $t$}} ; q )_\infty} \right]
        \\
        &
        \sum_{\ell \in \mathbb{Z}} \frac{1}{( \zeta q^{ \ell } ; q )_\infty} \mathbb{P}\left( \text{$n$-th particle at time $t$ is at position $\ell$}  \right)
        \\
        & = \prod_{i=1}^{n-1} \frac{(q;q)_\infty}{(a_i b_i;q)_\infty} 
        \left\llangle \frac{ H_q( \Z_n ; \Y_t ) H_q( \X_n ; \B_t ) }{ H_q( \Z_n ; \B_t ) H_q( \X_n ; \Y_t ) } \middle| \frac{(\zeta XZ,q/\zeta XZ,q;a)_\infty}{(\zeta,q/\zeta;q)_\infty} \frac{ H_q(\X_n;\Z_n) H_q(\A_{n-1};\B_{n-1}) }{H_q(\X_n;\B_{n-1}) H_q(\A_{n-1};\Z_{n})} \right\rrangle_n.
    \end{split}
\end{equation}
For a derivation, see \cite[Section 4]{imamura2017fluctuations}. Expressions of the kind \eqref{eq:laplace transform} have appeared in literature (in simpler forms) in \cite{BorodinCorwinRemenik,thieryLD2015integrable,imamura2019stationary,IMS_matching}, and they are useful as they can be turned, through algebraic manipulations, into determinantal formulas that are amenable to asymptotic analysis. Furthermore they provide connections between particle systems and explicit determinantal point processes \cite{IMS_matching,CafassoClaeysBiorthogonal}. This paper provides a further hint that such formulas are also related to orthogonality and completeness properties of special families of symmetric polynomials which describe the joint law of the stochastic system. 
Interestingly not all formulas of the kind \eqref{eq:laplace transform} have been linked yet to determinantal point processes or to families of orthogonal symmetric functions. In \cite{thieryLD2015integrable}, finite dimensional representation for Laplace trasform of a directed polymer in a random environment with Inverse Beta weights, have been conjectured. Despite the appearence of the Beta function in the model described in \cite{thieryLD2015integrable} and in the spin Whittaker functions, the latter appear to not be the naturally linked to the Inverse Beta Polymer, showing that the study of integral formulas of the kind \eqref{eq:laplace transform} might still bring insight to the theory of symmetric functions.

\bibliographystyle{alpha}
\bibliography{bib}

\end{document}